\documentclass[english,reqno]{amsart}
\usepackage[T1]{fontenc}
\usepackage[latin9]{inputenc}
\usepackage{verbatim}
\usepackage{amsthm}
\usepackage{amstext}
\usepackage{amssymb}

\makeatletter
\numberwithin{equation}{section}
\numberwithin{figure}{section}
\theoremstyle{plain}
\newtheorem{thm}{\protect\theoremname}[section]
  \theoremstyle{definition}
  \newtheorem{defn}[thm]{\protect\definitionname}
  \theoremstyle{remark}
  \newtheorem{rem}[thm]{\protect\remarkname}
  \theoremstyle{plain}
  \newtheorem{cor}[thm]{\protect\corollaryname}
  \theoremstyle{plain}
  \newtheorem{prop}[thm]{\protect\propositionname}
  \theoremstyle{plain}
  \newtheorem{lem}[thm]{\protect\lemmaname}
  \theoremstyle{definition}
  \newtheorem{example}[thm]{\protect\examplename}
\newenvironment{lyxlist}[1]
{\begin{list}{}
{\settowidth{\labelwidth}{#1}
 \setlength{\leftmargin}{\labelwidth}
 \addtolength{\leftmargin}{\labelsep}
 }}
{\end{list}}

\usepackage{amscd}

\theoremstyle{plain}

\newcommand{\N}{\mathbb{N}}
\newcommand{\R}{\mathbb{R}}

\newcommand{\Z}{\mathbb{Z}}

\newcommand{\rad}[1]{\mathrm{rad}_{#1\to\infty}}

\makeatother

\usepackage{babel}
  \providecommand{\corollaryname}{Corollary}
  \providecommand{\definitionname}{Definition}
  \providecommand{\examplename}{Example}
  \providecommand{\lemmaname}{Lemma}
  \providecommand{\propositionname}{Proposition}
  \providecommand{\remarkname}{Remark}
\providecommand{\theoremname}{Theorem}

\begin{document}

\title[Concentration analysis]{Concentration analysis in Banach spaces}

\author{Sergio Solimini}

\address{Politecnico di Bari, via Amendola, 126/B, 70126 Bari, Italy}

\email{sergio.solimini@poliba.it}

\author{Cyril Tintarev}

\address{Uppsala University, P.O.Box 480, 751 06 Uppsala, Sweden}

\email{tintarev@math.uu.se}

\thanks{This author thanks the mathematics department of Politecnico di Bari,
as well as of Bari University, for their warm hospitality. }


\begin{abstract}
The concept of a profile decomposition formalizes concentration compactness
arguments on the functional-analytic level, providing a powerful refinement
of the Banach-Alaoglu weak-star compactness theorem. We prove existence
of profile decompositions for general bounded sequences in uniformly
convex Banach spaces equipped with a group of bijective isometries,
thus generalizing analogous results previously obtained for Sobolev
spaces and for Hilbert spaces. Profile decompositions in uniformly
convex Banach spaces are based on the notion of $\Delta$-convergence
by T. C. Lim \cite{Lim} instead of weak convergence, and the two
modes coincide if and only if the norm satisfies the well-known Opial
condition, in particular, in Hilbert spaces and $\ell^{p}$-spaces,
but not in $L^{p}(\R^{N})$, $p\neq2$. $\Delta$-convergence appears
naturally in the context of fixed point theory for non-expansive maps.
The paper also studies connection of $\Delta$-convergence with Brezis-Lieb
Lemma and gives a version of the latter without an assumption of convergence
a.e.
\end{abstract}
\maketitle

\section{\label{intro}Introduction.}

Finding solutions of equations in functional spaces, in particular
of differential equations, typically involves the question of convergence
of functional sequences, which in turn often relies on compactness
properties of the problem. At the same time, infinite-dimensional
Banach spaces have no local compactness. Lack of compactness in a
sequence can be qualified in a variety of ways. For example, one can
look for coarser topologies in which sequences of particular type,
bounded in norm, become relatively compact. Banach-Alaoglu theorem
assures that a closed ball in any Banach space is compact in the weak{*}
topology. Concentration compactness principle (put in the terms of
Willem and Chabrowski - see the presentation in \cite{Chab} ) addresses
the situation when the norm in a functional space is expressed by
means of integration of some measure-valued map, which we may call
a Lagrangean, and when the Lagrangean, evaluated on a given sequence,
has a weak measure limit, its singular support is called a concentration
set. For specific sequences it is then possible to show that the singular
part of the limit measure is zero, which typically yields convergence
of the sequence in norm. As an illustration we sketch an argument
for existence of minimizers (Maz'ya \cite{Mazya}, Talenti \cite{Talenti})
in the Sobolev inequality:
\[
0<S_{N,p}=\inf_{u\in\dot{W}^{1,p}(\R^{N}):\,\|u\|_{p^{*}}=1}\int|\nabla u|^{p}\mathrm{d}x,\quad N>p\ge1,\quad p^{*}=\frac{pN}{N-p}\;.
\]
The Sobolev imbedding $\dot{W}^{1,p}(\R^{N})\hookrightarrow L^{p^{*}}(\R^{N})$
is not compact. It is invariant, however, with respect to transformations
$g[j,y](x)=2^{j\frac{N-p}{p}}u(2^{j}(x-y))$, $j\in\Z$, $y\in\R^{N}$,
and furthermore, if for any sequence $(j_{k},y_{k})\subset\Z\times\R^{N}$
one has $g[j_{k},y_{k}]u_{k}\rightharpoonup0$, then $u_{k}\to0$
in $L^{p^{*}}$ (Lions, \cite{Lions}). Therefore, if $\|u_{k}\|_{\dot{W}^{1,p}}^{p}\to S_{N,p}$
while $\|u_{k}\|_{p^{*}}=1$, then, necessarily, there exist $(j_{k},y_{k}){}_{k\in\N}\subset\Z\times\R^{N}$
such that a renamed subsequence of $g[j_{k},y_{k}]u_{k}$, which we
denote $v_{k}$ (and which, by invariance, is a minimizing sequence
as well) converges weakly to some $w\neq0$. A further reasoning that
involves convexity may be then employed to show that $\limsup\|u_{k}\|_{\dot{W}^{1,p}}^{p}<S_{N,p}$
unless $\|w\|_{p^{*}}=1$, and thus $w$ is a minimizer.

Concentration arguments in application to variational analysis of
PDE were developed and applied in the 1980's in works of Uhlenbeck,
Brezis, Coron, Nirenberg, T. Aubin, Lieb, Struwe, and P.-L. Lions,\textbf{
}with perhaps the most notable application being the Yamabe problem
of prescribed mean curvature \cite{Aubin,Trudinger,Schoen}. This
classical concentration compactness stimulated development of a more
detailed analysis of loss of compactness in terms of \textit{profile
decompositions}, starting with the notions of \emph{global compactness}
(for bounded domains and critical nonlinearities) of Struwe (\cite{Struwe})
and of translational compactness in $\mathbb{R}^{N}$ (for subcritical
nonlinearities) of Lions (1986). We do not aim here to provide a survey
of concentration compactness and its applications over the decades,
and refer the reader instead to the monographs of Chabrowski \cite{Chab}
and Tintarev \& Fieseler \cite{Fieseler-Tintarev}. 

A systematic \emph{concentration analysis} extends the concentration
compactness approach, from particular types of sequences in functional
spaces to \emph{general} sequences in functional spaces, and, further,
to general sequences in Banach spaces, studied in relation to general
concentration mechanisms modelled as actions of non-compact operator
groups. Concentration analysis can thus avail itself to the methods
of wavelet analysis: when the group, responsible for the concentration
mechanism, generates a wavelet basis, concentration may be described
in terms of sequence spaces of the wavelet coefficients. The counterpart
of Fourier expansion in the concentration analysis is \emph{profile
decomposition}. A profile decomposition represents a given bounded
sequence as a sum of its weak limit, decoupled elementary concentrations,
and a remainder convergent to zero in a way appropriate for some significant
application. An \emph{elementary concentration} is a sequence $g_{k}w\rightharpoonup0$,
where $g_{k}$ is a sequence of transformations (\emph{dislocations})
involved in the loss of compactness and the function $w$ is called
a \emph{concentration profile}. The 1995 paper of Solimini \cite{Solimini}
has shown that in case of Sobolev spaces any bounded sequence admits
a profile decomposition, with the only difference from the Palais-Smale
sequences for semilinear elliptic functionals (for which such profile
decompositions were previously known) being that it might contain
countably many, rather than finitely many, decoupled elementary concentrations.
The work of Solimini was independently reproduced in 1998-1999 by
Gérard \cite{Gerard} and Jaffard \cite{Jaffard}, who, on one hand,
provided profile decompositions for fractional Sobolev spaces as well,
but, on the other hand, gave a weaker form of remainder. It was subsequently
realized by Schindler and Tintarev \cite{SchindlerTintarev}, that
the notion of a profile decomposition can be given a functional-analytic
formulation, in the setting of a general Hilbert space and a general
group of isometries (see the alternative proof by Terence Tao via
non-standard analysis in \cite{TaoBlog2} p.168 ff.) This, in turn
stimulated the search for new concentration mechanisms, which included
inhomogeneous dilations $j^{-1/2}u(z^{j}$), $j\in\N$, with $z^{j}$
denoting an integer power of a complex number, for problems in the
Sobolev space $H_{0}^{1,2}(B)$ of the unit disk, related to the Trudinger-Moser
functionals \cite{AT}; and the action of the Galilean invariance,
together with shifts and rescalings, involved in the loss of compactness
in Strichartz imbeddings for the nonlinear Schrödinger equation (see
\cite{Tao,KiVi}). The existence of profile decompositions involving
the usual rescalings and shifts was established by Kyriasis \cite{Kyriasis}
Koch \cite{Koch}, Bahouri, Cohen and Koch \cite{BCK} for imbeddings
involving Besov, Triebel-Lizorkin and BMO spaces, although, like all
similar work based on the use of wavelet bases, it provided only a
weak form of remainder. Related results involving Morrey spaces were
obtained recently in \cite{PP}, using, like here, more classical
decompositions of spaces instead of wavelets. We refer the reader
for details to the recent survey of profile decompositions, \cite{survey}.

What is obviously missing in all the prior literature are results
about the existence of profile decompositions in the general setting
of abstract Banach spaces. This paper introduces a general theory
of concentration analysis in Banach space, as a sequel to an earlier
Hilbert space version \cite{SchindlerTintarev} and similar results
for Sobolev spaces \cite{Solimini,AT} as well as their wavelet-based
counterparts for Besov and Triebel-Lizorkin spaces \cite{BCK}. The
difference between the Hilbert space and the Banach space case is
essential and is rooted in the hitherto absence, in a general Banach
space, of a simple energy inequality that controls the total bulk
of profiles. Our approach to convergence in this paper is based in
finding clusters of concentations with prescribed energy bounds, as
there is no transparent relation between the total concentration energy
and energies of elementary concentrations. Energy estimates that we
obtain are not optimal and are based on the modulus of convexity. 

In order to obtain such inequality we have had to abandon weak convergence
in favor of\emph{ $\Delta$-convergence} introduced by T. C. Lim \cite{Lim},
see Definition {\large{}\ref{def: polar_limit}} below. The definition
applies to metric spaces as well, and are considered in this more
general setting in \cite{st3}. The notion of $\Delta$-limit is connected
to the notion of asympotic center of a sequence (\cite{Edelstein},
see Appendix B in this paper), namely, a sequence is $\Delta$-convergent
to $x$ if $x$ is an asymptotic center for each of its subsequences.
In Hilbert spaces, $\Delta$-convergence and weak convergence coincide
(this can be observed following the calculations in a related statement
of Opial, \cite[Lemma 1]{Opial}). More generally, the classical Opial's
condition (Condition 2 of \cite{Opial}, see Definition \ref{def:StrongOpial})
that has been in use for decades in the fixed point theory, is equivalent,
for uniformly convex and uniformly smooth spaces, to the condition
that $\Delta$-convergence and weak convergence coincide. Opial's
condition, however, does not hold in $L^{p}(\R^{N})$-spaces unless
$p\neq2$, as shown in \cite{Opial}. 

Similarly to the Banach-Alaoglu theorem, every bounded sequence in
the uniformly convex Banach space has a $\Delta$-convergent subsequence,
which follows from the $\Delta$-compactness theorem of Lim (\cite[Theorem 3]{Lim}).

\subsection*{Role of $\Delta$-convergence in profile decompositions}

It is shown in \cite{SchindlerTintarev} that any bounded sequence
in a Hilbert space $H$, equipped with an appropriate group $D$ of
isometries (called \emph{dislocations} or \emph{gauges}), has a subsequence
consisting of a sum of asymptotically orthogonal \textbf{\emph{elementary
concentrations}} and a remainder that \textbf{\emph{converges to zero}}\textbf{
}\textbf{\emph{$D$-weakly}}. These two terms mean the following.
\emph{An elementary concentration} (sometimes called a\emph{ bubble}
or a \emph{blow-up}) is an expression of the form $g_{k}w$, $k\in\N$,
where $w\in H$ (called \emph{concentration profile}) and $(g_{k})\subset D$
is a sequence weakly convergent to zero in the operator sense, such
that $g_{k}^{-1}u_{k}\rightharpoonup w$. A sequence $(u_{k})\subset H$
is \emph{convergent $D$-weakly to zero} if for any sequence $(g)\subset D$,
$g_{k}u_{k}$ converges to zero weakly. $D$-weak convergence is generally
stronger than weak convergence, and, in important applications, it
implies convergence in the norm of some space $X$ for which the imbedding
$H\hookrightarrow X$ is not compact. Profile decompositions with
a remainder vanishing despite the non-compactness of an imbedding
express \emph{defect of compactness} for a sequence, in form of a
rigidly structured sum of elementary blowups. Of course, vanishing
of the remainder in a useful norm depends on an appropriate choice
of group $D$. It is easy to check that if $D$ consists of all unitary
operators, $D$-weak convergence becomes norm convergence, and if
$D$ is compact, $D$-weak convergence coincides with weak convergence.
A useful group $D$ lies somewhere between these extremes. A continuous
imbedding $H\hookrightarrow X$ is called \emph{cocompact relative
to the group $D$ }if any $D$-weakly convergent sequence in $H$
is convergent in the norm of $X$. The notion of cocompactness extends
naturally to Banach spaces, but the proof for the profile decomposition
in Hilbert spaces cannot be generalized to the case of Banach spaces,
as the summary bulk of concentration profiles of a sequence $u_{k}$
is controlled by the inequality
\[
\sum_{n\in\N}\|w^{(n)}\|^{2}\le\liminf\|u_{k}\|^{2}.
\]
This inequality is, in turn, a consequence of the elementary relation
\begin{equation}
u_{k}\rightharpoonup u\,\Longrightarrow\|u_{k}\|^{2}=\|u_{k}-u\|^{2}+\|u\|^{2}+o(1).\label{eq:brl2}
\end{equation}
(For convenience of presentation, in equalities and inequalities between
terms of real-valued sequences, we use, as long as it does not cause
ambiguity, a Bachmann\textendash Landau notation o(1) to denote a
sequence of real numbers convergent to zero. In other words, $a_{k}=b_{k}+o(1)$
stays for $\lim_{k\to\infty}(a_{k}-b_{k})=0$, and $a_{k}\le b_{k}+o(1)$,
$a_{k},b_{k}\in\R$, $k\in\N$, stays for $\limsup_{k\to\infty}(a_{k}-b_{k})\le0$.)
A plausible conjecture for the general uniformly convex Banach space
(see Appendix A for definitions, in particular of the modulus of convexity,
denoted as $\delta$) would be, assuming for simplicity that $\liminf\|u_{k}\|\le1$,
that 
\begin{equation}
\sum_{n\in\N}\delta(\|w^{(n)}\|)\le\liminf\|u_{k}\|,\label{eq:nrg}
\end{equation}
where $\delta$ is the modulus of convexity for $X$. We have however,
as the closest Banach space version of (\ref{eq:brl2}), the inequality

\begin{equation}
u_{k}\rightharpoonup u\,\Longrightarrow\|u_{k}\|\ge\|u\|+\delta(\|u_{k}-u\|)+o(1),\label{eq:wl}
\end{equation}
proving which is an easy exercise using the definition of uniform
convexity and weak lower semicontinuity of the norm that we leave
to the reader. On the other hand, a desired inequality that leads
to (\ref{eq:nrg}) is rather 
\begin{equation}
\|u_{k}\|\ge\|u_{k}-u\|+\delta(\|u\|)+o(1),\label{eq:pl}
\end{equation}
and it is generally false when $u_{k}\rightharpoonup u$. It is true,
however, if $u$ is a $\Delta$-limit, rather than a weak limit of
$u_{k}$ (see Lemma \ref{lem:energy} below). In other words, concentration
profiles for sequences in Banach spaces emerge not as weak limits
of ``deflation'' sequences $g_{k}^{-1}u_{k}$, but as their $\Delta$-limits. 

\vspace{1cm}

We restrict consideration of Banach spaces to the class of uniformly
convex spaces, as the natural next step after having studied matters
of weak convergence and profile decomposition in Hilbert spaces. Uniformly
convex spaces have many common properties with Hilbert spaces, in
particular, reflexivity, Kadec property ($u_{k}\rightharpoonup u$,
$\|u_{k}\|\to\|u\|$ $\Longrightarrow$ $\|u_{k}-u\|\to0$), uniqueness
of $\Delta$-limits and sequential $\Delta$-compactness of balls,
that general Banach spaces do not necessarily possess.

It appears that sharper than (\ref{eq:wl}) or (\ref{eq:pl}) lower
bounds for the norms of sequences in Banach spaces require the use
of both the weak limit and the $\Delta$-limit. Among these cases
there is the important Brezis-Lieb Lemma (\cite{Brezis-Lieb}), which
states that if $(\Omega,\mu)$ is a general measure space and $u_{k}\rightharpoonup u$
in $L^{p}(\Omega,\mu)$, $1\le p<\infty$, and $u_{k}\to u$ $\mu$-a.e.
in $\Omega$, then 
\begin{equation}
\|u_{k}\|_{L^{p}}^{p}=\|u_{k}-u\|_{L^{p}}^{p}+\|u\|_{L^{p}}^{p}+o(1)\,.\label{eq:SuggestedNewEquation}
\end{equation}
Remarkably, no a.e. convergence is required for (\ref{eq:SuggestedNewEquation})
to hold when $\mu$ is a counting measure or when $p=2$ (when it
follows from (\ref{eq:brl2})). If, however, one does not assume convergence
a.e., one has the following analog of Brezis-Lieb lemma, proved in
Section 5 for $p\ge3$, namely an expression for a lower bound for
the norm of the sequence 
\begin{equation}
\|u_{k}\|_{L^{p}}^{p}\ge\|u_{k}-u\|_{L^{p}}^{p}+\|u\|_{L^{p}}^{p}+o(1)\label{eq:WBL}
\end{equation}
 where $u$ is assumed to be both the weak limit and the $\Delta$-limit
of the sequence (but no a.e. convergence is assumed). It is shown
in \cite{AditiBL} that condition $p\ge3$ is necessary, in particular,
when $(\Omega,\mu)$ is an interval with the Lebesgue measure.

The paper is organized as follows. In Section 2 we give the precise
definitions of the concepts arising in concentration analysis and
formulate our main results. Section 3 studies basic properties of
$\Delta$-convergence in uniformly convex Banach spaces. In Section
4 we prove the inequality (\ref{eq:WBL}). In Section 5 we prove the
existence of an abstract profile decomposition in terms of $\Delta$-convergence,
for every bounded sequence in a uniformly convex and uniformly smooth
Banach space, whenever the relevant collection of bijective isometries
on $X$ satisfies appropriate hypotheses. It is important to note
that the argument for existence of profile decomposition in Banach
spaces is different both from the Sobolev space case (\cite{Solimini}
), where the norms show a natural asymptotic decoupling behavior with
regard to distinct rescalings and from the general Hilbert spaces
case (\cite{SchindlerTintarev}), where decoupling of distinct concentrations
is expressed by their asymptotic orthogonality. In Section 6 we give
a general discussion of cocompactness and related properties. In Section
7 we prove Theorem \ref{thm:main-1}, discuss the remainder of the
profile decomposition in the context of cocompact imbeddings, and
give examples of the latter. In Appendix A we list definitions and
elementary properties of uniformly convex and uniformly smooth Banach
spaces, and in Appendix B we present the notion of asymptotic center
and its connection to $\Delta$-convergence. Appendix C discusses
an equivalent form of the main condition for the groups involved in
profile decompositions. 

The main results of the paper are:
\begin{itemize}
\item Profile decompoisitions: Theorem \ref{thm:main}, its simplified version
Theorem \ref{thm:main-1},  and profile decomposition in the dual
space: Proposition \ref{prop:dual coco} and Theorem\ref{thm:dual coco};
\item Equivalence of the classical Opial's condition in uniformly convex
and uniformly smooth spaces to the property that weak convergence
and $\Delta$-convergence coincide, Theorem \ref{thm:UCOpial};
\item An analog of the Brezis-Lieb lemma, where the assumption of pointwise
convergence replaced by the assumption of equal weak and $\Delta$-limits,
Theorem \ref{thm:newbl}.
\end{itemize}

\section{Basic notions of concentration analysis and statement of results}

The key element required for obtaining a cocompact imbedding of a
Banach space $X$ into a Banach space $Y$ is a collection $D$ of
operators which act isometrically and surjectively (and thus bijectively)
on $X$ and which are chosen in such a way that any bounded sequence
of elements in $X$ which convergence weakly to zero under action
of any sequence from $D$ (see Definition \ref{(Gauged-weak-convergence)}
below) must converge to zero in the norm of $Y$. The operators of
$D$ are often referred to as ``blow-up'' or ``rescaling'' isometries
since a frequently occurring example of $D$ is the set of typical
concentration actions $u\mapsto t^{r}u(t\cdot)$, $t>0$. It seems
better, however, to use some more general terminology, such as \emph{gauges,}
or\emph{ dislocations,} to refer to these operators, since $D$ can
be quite different in other important cases. For example, it may consist
of actions of anisotropic or inhomogeneous dilations, of isometries
on Riemannian manifolds, or of shifts in the Fourier variable. An
elementary example is provided by a set of index shifts $u\mapsto u_{\cdot+j}$
on a sequence space.%
\begin{comment}
Let $Z$ be a Banach space. One says that a linear operator $h$ on
$Z$ is a bijective quasi-isometry if $h$ is bijective and both $h$
and $h^{-1}$ are bounded. 
\end{comment}

Let $D$ be a set of bijective isometries on a Banach space $Z$.
We will use the following notation: $D^{-1}=\{h^{-1}\}_{h\in D}$. 
\begin{defn}
\label{(Gauged-weak-convergence)}(Gauged weak convergence) Let $Z$
be a Banach space, and let $D\ni I$ be a bounded set of bijective
isometries on $Z$ such that $D^{-1}$ is also a bounded set. One
says that a sequence $(u_{k})_{k\in\mathbb{N}}$ of elements in $Z$
converges to zero $D$-weakly if $g_{k}^{-1}u_{k}\rightharpoonup0$
for every choice of the sequence $(g_{k}){}_{k\in\mathbb{N}}\subset D$.
We use the notation $u_{k}\stackrel{D}{\rightharpoonup}0$ to denote
\emph{$D$-}weak convergence and the notation $u_{k}\stackrel{D}{\rightharpoonup}u$
to mean that $u_{k}-u\stackrel{D}{\rightharpoonup}0$. 
\end{defn}
We remark that in analogues of this definition appearing in earlier
papers on this subject, the roles of $D$ and $D^{-1}$ are interchanged.
This makes no difference when $D$ is a group. 

\begin{comment}
Note that $D$-weak convergence implies weak convergence and is implied
by convergence in norm. If $D=\{I\}$, $D$-weak convergence trivially
coincides with weak convergence. If $Z$ is a Hilbert space and $D$
is the group of all bijective isometries on $Z$, then $D$-weak convergence
coincides with convergence in norm.
\end{comment}
{} The definition below will be adapted to the different mode of convergence
introduced in the course of argument, but remains relevant for the
class of norms satisfying the Opial's condition which arise in most
known applications. 
\begin{defn}
\label{(Cocompact set)}(Cocompact subsets) Let $Z$ be a Banach space,
and let $D\ni I$ be a set of bijective isometries on $Z$. A set
$B\subset Z$ is called $D$-cocompact if every $D$-weakly convergent
sequence in $B$ converges in norm in $Z$. 
\end{defn}
Clearly the limit in norm of such a sequence must be the same element
as its $D$-weak limit. It is also clear that every precompact subset
of $X$ is also $D$-cocompact. 
\begin{defn}
\label{(Cocompactness)}(Cocompact imbeddings) Let $X$ be a Banach
space continuously embedded into a Banach space $Y$. Let $D\ni I$
be a set of bijective isometries on $X$. Suppose that every sequence
$(u_{k})_{k\in\mathbb{N}}$ satisfying $u_{k}\stackrel{D}{\rightharpoonup}0$
in $X$ also satisfies $\|u_{k}\|_{Y}\to0$. Then we say that the
imbedding $X\hookrightarrow Y$ is \emph{cocompact} relative to the
set $D$, and we denote this by writing $X\stackrel{D}{\hookrightarrow}Y$.
\end{defn}
It is easy to see that the following definition, under the additional
assumptions it makes, is equivalent to Definition \ref{(Cocompactness)}. 
\begin{defn}
\label{(Cocompactness)-1} Let $X$ be a Banach space continuously
and embedded into a Banach space $Y$ and assume that $X$ is dense
in $Y$ and $Y^{*}$ is dense in $X^{*}$. Let $D\ni I$ be a set
of bijective isometries on $Y$, and assume that the set $D_{X}$
of restrictions of operators in $D$ to $X$ defines a set of bijective
isometries on $X$. One says that the imbedding $X\hookrightarrow Y$
is cocompact\emph{ }relative to the set $D$, if all bounded subsets
of $X$ are $D$-cocompact in $Y$. 
\end{defn}
In what follows, weak convergence of a sequence of operators $(A_{k}){}_{k\in\mathbb{N}}$
on a Banach space $X$ to an operator $A$, i.e. $A_{k}x\rightharpoonup Ax$
for each $x\in X$, will be denoted by $A_{k}\rightharpoonup A$.
The following question arises immediately when one knows which set
of bijective isometries $D$ is responsible for concentration, or,
in other words, when an imbedding $X\hookrightarrow Y$ is cocompact
relative to a given set $D$: Is it possible, for any bounded sequence
in $X$, to produce a subsequence which is norm convergent in $Y$
by subtraction of \emph{elementary concentrations}? We recall that
by an elementary concentration for a sequence $(u_{k})_{k\in\mathbb{N}}\subset X$
we mean a sequence of the special form $(g_{k}w)_{k\in\mathbb{N}}$,
where $(g_{k})_{k\in\mathbb{N}}\subset D$, $g_{k}\rightharpoonup0$,
and $g_{k}^{-1}u_{k}\rightharpoonup w\neq0$ in $X$ on some renamed
subsequence. The use of word \emph{concentration} originates in the
case when the set $D$ consists of dilation operators on a functional
space, so that, as $k$ tends to $\infty$, the graphs of the functions
$g_{k}w$ become taller and narrower peaks clustering around some
point of the underlying set. Such concentrations occur in scale-invariant
PDE, such as semilinear elliptic equations with critical nonlinearities.
\begin{defn}
\label{def:pd}One says that a bounded sequence $(u_{k})_{k\in\mathbb{N}}$
in a Banach space $X$ admits a \emph{profile decomposition} with
respect to the set of bijective linear isometries $D\ni I$, if there
exists a sequence $r_{k}\stackrel{D}{\rightharpoonup}0$ and, for
each $n\in\mathbb{N}$, there exists an element $w^{(n)}\in X$ and
a sequence $(g_{k}^{(n)})_{k\in\mathbb{N}}\subset D$ such that $g_{k}^{(1)}=I$
and

\begin{equation}
(g_{k}^{(n)})^{-1}g_{k}^{(m)}\rightharpoonup0\text{ whenever }m\neq n\text{ (asymptotic decoupling of gauges)},\label{eq:decouple-1-1}
\end{equation}
and such that a renamed subsequence of $(u_{k})_{k\in\mathbb{N}}$
can be represented in the form 
\begin{equation}
u_{k}=\sum_{n=1}^{\infty}g_{k}^{(n)}w^{(n)}+r_{k}\text{ for each }k,\label{eq:ProDec-1-1}
\end{equation}
where the series $\sum_{n=1}^{\infty}g_{k}^{(n)}w^{(n)}$ is convergent
in $X$ unconditionally and uniformly in $k$. \textit{\emph{(It follows
immediately then that $(g_{k}^{(n)})^{-1}u_{k}\rightharpoonup w^{(n)},\; n\in\mathbb{N}$.)}}
\end{defn}
Note that in general, any subset of \emph{profiles} $w^{(n)}$ may
consist of zero elements. In particular, the sum in (\ref{eq:ProDec-1-1})
may be finite. 

In the Banach space setting (restricted in the present study to uniformly
convex spaces) we will establish the existence of a variant of this
profile decomposition, based on $\Delta$-convergence, studied in
the next section. $\Delta$-convergence, as we show in Theorem \ref{thm:UCOpial}
below, coincides with weak convergence if and only if the Opial's
condition (see e.g. Definition \ref{def:StrongOpial}) holds.

Our main result follows below. It uses a technical condition (\ref{eq:*-1})
that extends to the Banach space case the condition of dislocation
group used in \cite{SchindlerTintarev} for Hilbert space case, and
it is verified in a great number of applications. We refer the reader
for details to the book \cite{Fieseler-Tintarev} and to the recent
survey \cite{survey}. Our principle example of the class of spaces
that satisfy conditions of two theorems below is Besov spaces $\dot{B}{}^{s,p,q}\left(\mathbb{R}^{N}\right)$
and Triebel-Lizorkin spaces $\dot{F}{}^{s,p,q}\left(\mathbb{R}^{N}\right)$
with $s\in\R$ and $p,q\in(1,\infty)$ when supplied with equivalent
norms, based on Littlewood-Paley decomposition (see e.g. books of
Triebel \cite{Triebel} or Adams \& Fournier \cite{Adams}).
\begin{thm}
\label{thm:main-1}Let $X$ be a uniformly convex and uniformly smooth
Banach space that satisfies the Opial's condition. Let $D_{0}$ be
a group of linear isometries satisfying the property
\begin{equation}
(g_{k})\subset D_{0},\; g_{k}\not\rightharpoonup0\Longrightarrow\exists(k_{j})\subset\mathbb{N}:\;(g_{k_{j}}^{-1}),(g_{k_{j}})\text{ converge strongly (i.e. pointwise)}\label{eq:*-1}
\end{equation}
and let $D\ni I$ be a subset of $D_{0}$. Then every bounded sequence
$(u_{k})\subset X$ admits a profile decomposition with respect to
$D$. Moreover, if $\|u_{k}\|\le1$ for all $k$, and $\delta$ is
the modulus of convexity of $X$, then 
\begin{equation}
\limsup\|r_{k}\|+\sum_{n}\delta(\|w^{(n)}\|)\le1.\label{eq:Energy-1}
\end{equation}
where $r_{k}$ and $w^{(n)}$ are the elements arising in the profile
decomposition as defined in (\ref{eq:ProDec-1-1}).\end{thm}
\begin{rem}
The restriction $\|u_{k}\|\le1$ is inessential. Unless $(x_{k})$
has a subsequence convergent to zero in $X$ (in which case Theorem
\ref{thm:main} holds with $w^{(n)}=0$ for all $n$), one can apply
Theorem \ref{thm:main} to a subsequence of $x_{k}/\|x_{k}\|$ with
$\|x_{k}\|\to\nu>0$. Then the assertion of Theorem \ref{thm:main-1}
(and analogous statements further in this paper) will hold with the
only modification being $\delta$ replaced by $\nu\delta(\frac{\cdot}{\nu})$. 
\end{rem}

\begin{rem}
The assumption of uniform convexity cannot be removed, as we can see
from the example of $X=L^{\infty}(\mathbb{R})$ with $D$ being a
group of integer shifts. Let $x_{k}$ be a characteristic function
of a disjoint union of all intervals of the length $j/2^{k}$, $j=1,\dots2^{k}$,
translated in such a manner that the distance between any two intervals
exceeds $k$. Then the distinct profiles of $x_{k}$ will be characteristic
functions of all intervals $(0,t)$ , $t\in(0,1]$, and thus form
an an uncountable set. \end{rem}
\begin{cor}
If, in addition to the assumptions of Theorem \ref{thm:main-1}, the
space $X$ is $D$-cocompactly imbedded into another Banach space
$Y$, then the remainder $r_{k}$ converges to zero in the norm of
$Y$. 
\end{cor}
In the main body of the paper we first prove a more general statement,
Theorem \ref{thm:main}, similar to Theorem \ref{thm:main-1}, that
does not assume the Opial's condition, and then derive Theorem \ref{thm:main-1}
from it as an elementary corollary. In absence of the Opial's condition,
the argument is based on $\Delta$- and $D$-$\Delta$-convergence
instead of, respectively, weak and $D$-weak convergence. 

We also prove a conjecture by Michael Cwikel (personal communication)
that when $X\stackrel{D}{\hookrightarrow}Y$, the existence of profile
decompositions in $X$ implies the existence of ``dual'' profile
decompositions in $Y^{*}\stackrel{D^{\#}}{\hookrightarrow}X^{*}$,
where $ $
\[
D^{\#}=\{g{}^{*-1},\; g\in D\}.
\]

\begin{thm}
\label{thm:dual coco}Let $Y$ be a uniformly convex and uniformly
smooth Banach space that satisfies the Opial's condition. Let $I\in D\subset D_{0}$
where $D_{0}$ is a group of linear isometries in $X$ and $Y$ satisfying
(\ref{eq:*-1}). If $X\stackrel{D}{\hookrightarrow}Y$ and $X$ is
dense in $Y$, then $Y^{*}\stackrel{D^{\#}}{\hookrightarrow}X^{*}$
and any bounded sequence in $Y^{*}$ has a profile decomposition relative
to $D^{\#}$ with the remainder sequence \textup{$(r_{k})_{k\in\mathbb{N}}$}
converging in norm to 0 in $X^{*}$.
\end{thm}

\section{\label{sec:BasicNotions} $\text{\ensuremath{\Delta}}$-convergence
in uniformly convex spaces}

\subsection{Definition and basic properties}
\begin{defn}
\label{def: polar_limit}Let $(x_{k})_{k\in\mathbb{N}}$ be a sequence
in a Banach space $X$. One says that $x$ is a\emph{ $\Delta$-limt}
of $(x_{k})$ if 
\begin{equation}
\forall y\text{\ensuremath{\in}X}\quad\|x_{k}-x\|\text{\ensuremath{\le}}\|x_{k}-y\|+o(1).\label{eq:oldmw-1-1}
\end{equation}
We will use the notation $x_{k}\rightharpoondown x$ as well as $x=\mathrm{\stackrel{\rightharpoondown}{\lim}}\, x_{k}$
to denote $\Delta$-convergence. Obviously if $x_{k}$ converges to
$x$ in norm, then $x$ is a unique\emph{ $\Delta$}-limit of $(x_{k})$. \end{defn}
\begin{prop}
\label{thm:3eqdef}Suppose that $(x_{k})_{k\in\mathbb{N}}$ is a bounded
sequence in a uniformly convex Banach space $X$ and let $x\in X$.
If $x_{k}\rightharpoondown x$, then for each element $z\in X$ with
$z\ne x$ there exist a positive constant $k_{0}$ and a positive
constant $c$ depending on $\|x-z\|$ and $\sup_{k\in\mathbb{N}}\left\Vert x_{k}\right\Vert $
continuously in $(0,\infty)\times[0,\infty]$, such that

\begin{equation}
\left\Vert x_{k}-x\right\Vert \le\left\Vert x_{k}-z\right\Vert -c\mbox{ for all \ensuremath{k\ge k_{0}}}.\label{eq:NewStrongDom}
\end{equation}
\end{prop}
\begin{proof}
Given an element $z\ne x$ we first observe that $\liminf_{k\to\infty}\left\Vert x_{k}-z\right\Vert $
must be strictly positive, since otherwise there would be a subsequence
of $\left\{ x_{k}\right\} $ converging in norm to $z$. 

Without loss of generality we may assume that $\|x_{k}-x\|<1$ and
note that it suffices to prove (\ref{eq:NewStrongDom}) for $\|x-z\|<2$.
By uniform convexity, and taking into account (\ref{eq:oldmw-1-1}),
we have

\[
\|x_{k}-x\|\le\|x_{k}+\frac{1}{2}(x+z)\|+o(1)=
\]
\[
\|\frac{1}{2}[(x_{k}-x)+(x_{k}-z)]\|+o(1)\le\|x_{k}-z\|-\delta(\|x-z\|)+o(1),
\]
from which (\ref{eq:NewStrongDom}) is immediate. \end{proof}
\begin{cor}
The $\Delta$-limit in a uniformly convex Banach space is unique. 
\end{cor}
It is shown in \cite{Edelstein} that uniformly convex Banach spaces
are asymptotically complete (a metric space is called asymptotically
complete if every bounded sequence in it has an asymptotic center,
see Appendix B). Since every bounded sequence in an asymptotically
complete metric space has a $\Delta$-convergent subsequence by \cite[Theorem 3]{Lim},
we have the following analog of Banach-Alaoglu theorem:
\begin{thm}
\label{thm:newBA}Let $X$ be a uniformly convex Banach space and
let $(x_{k})\subset X$ be a bounded sequence. Then $(x_{k})$ has
a $\Delta$-convergent subsequence.
\end{thm}

\subsection{Uniform boundedness theorem}

It is well-known that for every $x\in X\setminus\{0\}$ there exists
an element $x^{*}\in X$, called a conjugate of $x$, such that $\|x^{*}\|=1$
and $\langle x^{*},x\rangle=\|x\|$. 

If $X^{*}$ is strictly convex, namely, if
\[
\xi,\eta\in X^{*},\xi\neq\eta,\text{\;}\|\xi\|=\|\eta\|=1,\;\Longrightarrow\|t\xi+(1-t)\eta\|<1\text{ for all }t\in(0,1),
\]
(in particular, when $X^{*}$ is uniformly convex or, equivalently,
when $X$ is uniformly smooth, see Appendix A), then the element $x^{\ast}$,
as one can immediately verify by contradiction, is unique. 
\begin{thm}
\label{thm:newBS}Let $X$ be a uniformly smooth and uniformly convex
Banach space, and let $(x_{k})\subset X$ be a $\Delta$-convergent
sequence. Then the sequence $(x_{k})$ is bounded.\end{thm}
\begin{proof}
It suffices to prove the theorem for the case $x_{k}\rightharpoondown0$,
since, once we prove that, from $x_{k}\rightharpoondown x$ follows
$x_{k}-x\rightharpoondown0$ and thus $x_{k}-x$ is bounded. Assume
that $\|x_{k}\|\to\infty$. Since $X$ is uniformly smooth, there
exists a function $\eta:[0,1]\to[0,\infty)$, $\text{\ensuremath{\lim}}_{t\to0}\eta(t)/t=0$,
such that (see \cite[p. 61]{LinTz}) 
\[
\left|\|x+y\|-\|x\|-\langle x^{*},y\rangle\right|\le\eta(\|y\|),\text{ whenever }\|x\|=1\text{ and }\|y\|\le1.
\]
  Then, using the notation $\omega(x,y)=\|x+y\|-\|x\|-\langle x^{*},y\rangle$,
we have 
\[
\|x+y\|^{2}-\|x\|^{2}=(\|x+y\|-\|x\|)(\|x+y\|-\|x\|+2\|x\|)=(\omega(x,y)+\langle x^{*},y\rangle)^{2}+2(\omega(x,y)+\langle x^{*},y\rangle).
\]
Substitute now $x=\frac{x_{k}}{\|x_{k}\|}$ and $y=\frac{z}{\|x_{k}\|}$
with an arbitrary vector $z$. Then, by Proposition \ref{thm:3eqdef},
we have 
\[
0\le\|x_{k}+z\|^{2}-\|x_{k}\|^{2}=\alpha_{k}^{2}+2\|x_{k}\|\alpha_{k}
\]
for all $k$ sufficiently large, where

\[
\alpha_{k}=\|x_{k}\|\omega(\frac{x_{k}}{\|x_{k}\|},\frac{z}{\|x_{k}\|})+\langle x_{k}^{*},z\rangle.
\]
Consequently, either $\alpha_{k}\ge0$ or $\alpha_{k}\le-2\|x_{k}\|\to-\infty$.
The latter case can be easily ruled out, since $\|x_{k}^{*}\|=1$,
$\langle x_{k}^{*},z\rangle$ is bounded, $\|x_{k}\||\omega(\frac{x_{k}}{\|x_{k}\|},\frac{z}{\|x_{k}\|})|\to0$
as $\|x_{k}\|\to\infty$, and so $\alpha_{k}$ is bounded. Therefore
we have necessarily, for large $k$, 
\[
\|x_{k}\|\omega(\frac{x_{k}}{\|x_{k}\|},\frac{z}{\|x_{k}\|})+\langle x_{k}^{*},z\rangle\ge0,
\]
and, thus, 
\[
\langle x_{k}^{*},z\rangle\ge-\eta(t_{k})/t_{k},
\]
 where $t_{k}=1/\|x_{k}\|$. In other words, we have $|\langle\psi(\|x_{k}\|)x_{k}^{*},z\rangle|\le1$,
for $k$ sufficiently large, where $\psi(t)=\frac{t^{-1}}{\eta(t^{-1})}$
satisfies $\psi(t)\to\infty$ when $t\to\infty$. By the Uniform Boundness
Principle the sequence $\psi(\|x_{k}\|)$ is bounded, but this contradicts
to the assumption $\|x_{k}\|\to\infty$, which proves the theorem.
\end{proof}
Note that without the condition of uniform smoothness, $\Delta$-convergent
sequences are not necessarily bounded. See \cite[Example 3.1]{st3}.

\subsection{Characterization of $\Delta$-convergence in terms of duality map}
\begin{lem}
\label{lem:additive} Let $X$ be a Banach space. If $(x_{k})_{k\in\mathbb{N}}$
is a bounded sequence, $x_{k}\rightharpoondown x$ and $y_{k}\to y$,
then $x_{k}+y_{k}\rightharpoondown x+y$.\end{lem}
\begin{proof}
It suffices to prove the assertion for $x=y=0$. let $z\in X$. Then

\[
\|x_{k}+y_{k}\|=\|x_{k}\|+o(1)\le\|x_{k}-z\|+o(1)=\|x_{k}+y_{k}-z\|+o(1),
\]
which proves the lemma.\end{proof}
\begin{lem}
\label{lem:energy} Let $X$ be a uniformly convex Banach space with
the modulus of convexity $\delta$. If $u_{k}\rightharpoondown u$
in $X$ and $\|u_{k}\|\le1$ for all $k\in\mathbb{N}$, then $\left\Vert u\right\Vert <2$
and, for all sufficiently large $k$, 
\begin{equation}
\|u_{k}\|\ge\|u_{k}-u\|+\delta(\|u\|).\label{eq:energy}
\end{equation}
\end{lem}
\begin{proof}
We can suppose that $u\ne0$ since the result is a triviality for
$u=0$. Note that for $k$ sufficiently large, $\|u_{k}-u\|<\|u_{k}\|$.
This inequality implies that $\left\Vert u\right\Vert <2\left\Vert u_{k}\right\Vert \le2$
and it also implies that $u_{k}\ne0$ for these values of $k$. Thus
we may apply (\ref{eq:UltraSuperUC}) with $C_{1}=\left\Vert u_{k}\right\Vert $
and $C_{2}=1$ to the elements $u_{k}$ and $u_{k}-u$ to obtain that
\[
\left\Vert u_{k}-\frac{1}{2}u\right\Vert =\left\Vert \frac{u_{k}+\left(u_{k}-u\right)}{2}\right\Vert \le\left\Vert u_{k}\right\Vert -\delta\left(\left\Vert u\right\Vert \right)\,.
\]

Finally, since $u_{k}\rightharpoondown u$, one also has $\|u_{k}-u\|\le\|u_{k}-\frac{1}{2}u\|$
for sufficiently large $k$ and (\ref{eq:energy}) follows.%
\begin{comment}
As a preamble to our next result we mention a related known result.
Our proof however will be independent of it. We recall that the duality
map $x\mapsto x^{*}$ is continuous from $X\setminus\{0\}$ equipped
with the norm topology to $X^{*}$ equipped with the weak-$*$ topology.
(See e.g. Proposition 2.1.5 in Clarke, \cite{Clarke}. The proposition
is stated in terms of the subdifferential set. For definition of the
subdifferential and the identification of the subdifferential set
of the the map $x\mapsto\|x\|$ with the set of conjugates $x*$ see
Ekeland and Temam \cite{ET}, Sections I.5.1-I.5.2.) 
\end{comment}

\end{proof}
We have the following characterization of $\Delta$-convergence by
means of the duality map $x\mapsto x^{\ast}$.
\begin{thm}
\label{thm:conj}Let $X$ be a uniformly convex and uniformly smooth
Banach space. Let $x\in X$ and let $(x_{k})_{k\in\N}\subset X$ be
a bounded sequence such that $\liminf\|x_{k}-x\|>0$. Then $x_{k}\rightharpoondown x$
if and only if $(x_{k}-x)^{*}\rightharpoonup0$. \end{thm}
\begin{proof}
Without loss of generality we  need only to consider the case $x=0$. 

\textbf{\emph{Sufficiency}}\textbf{.} Suppose that $x_{k}^{*}\rightharpoonup0$.
Then for any $y\in X$, $\left\langle x_{k}^{*},y\right\rangle \to0$
and so 
\[
\left\Vert x_{k}\right\Vert =\left\langle x_{k}^{*},x_{k}\right\rangle =\left|\left\langle x_{k}^{*},x_{k}-y\right\rangle +\left\langle x_{k}^{*},y\right\rangle \right|\le\left\Vert x_{k}-y\right\Vert +o(1)\,,
\]
 i.e $x_{k}\rightharpoondown0$.

\textbf{\emph{Necessity. }}Suppose that $x_{k}\rightharpoondown0$.
By Proposition \ref{thm:3eqdef}, for any $y\in X$, there exists
an integer $k(y)$ such that $\|x_{k}\|\le\|x_{k}-y\|$ for all $k\ge k(y)$.
Then

\begin{eqnarray*}
\|x_{k}\| & \le & \|x_{k}-y\|=\left\langle \left(x_{k}-y\right)^{*},x_{k}-y\right\rangle =\left\langle \left(x_{k}-y\right)^{*},x_{k}\right\rangle -\left\langle \left(x_{k}-y\right)^{*},y\right\rangle \\
 & \le & \left\Vert x_{k}\right\Vert -\left\langle \left(x_{k}-y\right)^{*},y\right\rangle ,
\end{eqnarray*}
Consequently we have 
\[
\left\langle \left(x_{k}-y\right)^{*},y\right\rangle \le0\mbox{ for all }k\ge k(y)\,.
\]
Since $\liminf\left\Vert x_{k}\right\Vert >0$, we may assume that
$k(y)$ is large enough so that $\left\Vert x_{k}\right\Vert \ge2\lambda$
for some positive constant $\lambda$ whenever $k\ge k(y)$. So, if
we consider only those $y$ which satisfy $\left\Vert y\right\Vert \le\lambda$
and those $x_{k}$ for which $k\ge k(y)$, we can assert that $x_{k}$
and $x_{k}-y$ are both contained in the set $E=\left\{ x\in X:\left\Vert x\right\Vert \ge\lambda\right\} $
and therefore deduce from Lemma \ref{lem:dualconv}, for each $\epsilon\in(0,1/4)$,
that $\left\Vert \left(x_{k}-y\right)^{*}-x_{k}^{*}\right\Vert \le\epsilon$
whenever $0<\left\Vert y\right\Vert \le\min\left\{ \frac{3\lambda}{2}\delta\left(\epsilon\right),\frac{\lambda}{2}\right\} $.
For such choices of $y$ we will therefore have 
\[
\left\langle x_{k}^{*},\frac{y}{\left\Vert y\right\Vert }\right\rangle \le\epsilon\mbox{ for all \ensuremath{k\ge k(y)}.}
\]
Applying the same reasoning to the element $-y$, we obtain that $\left|\left\langle x_{k}^{*},\frac{y}{\left\Vert y\right\Vert }\right\rangle \right|\le2\epsilon$
whenever $0<\left\Vert y\right\Vert \le\min\left\{ \frac{3\lambda}{2}\delta\left(\epsilon\right),\frac{\lambda}{2}\right\} $
and $k\ge k_{0}=\max\left\{ k(y),k(-y)\right\} $. In other words,
given any $w\in X$ with $\left\Vert w\right\Vert =1$, we know that
$\left|\left\langle x_{k}^{*},w\right\rangle \right|\le2\epsilon$
for all $k\ge k_{0}(w,\epsilon)$ for some sufficiently large $k_{0}(w,\epsilon)$.
Consequently, $x_{k}^{*}\rightharpoonup0$. \end{proof}
\begin{cor}
\label{cor:lp}Let $X$ be either a Hilbert space or the $\ell^{p}$-space
with $1<p<\infty$, and let $(x_{n})$ be a sequence in $X$. Then
$x_{n}\rightharpoonup x$ if and only if $x_{n}\rightharpoondown x$.\end{cor}
\begin{proof}
We may assume that $\liminf\|x_{n}-x\|>0$, since for subsequences
that converge to $x$ in norm the result is trivial.

Let $X$ be a Hilbert space and recall that we are using the definition
of conjugate dual with the unit norm. If $x_{n}\rightharpoonup x$,
then for any $y\in X$, 
\[
|(x_{n}-x)^{*},y)|=\left|\left(\frac{x_{n}-x}{\|x_{n}-x\|},y\right)\right|\le\frac{1}{\liminf\|x_{n}-x\|}\limsup\left|\left(x_{n}-x,y\right)\right|+o(1)\to0.
\]
Conversely, if $x_{n}\rightharpoondown x$, then for any $y\in X$,
taking into account that $(x_{n})$ is bounded by Theorem \ref{thm:newBS},
we have 
\[
|(x_{n}-x,y)|=\|x_{n}-x\||((x_{n}-x)^{*},y)|\le(\sup\|x_{n}\|+\|x\|)|((x_{n}-x)^{*},y)|\to0.
\]

Let now $X=\ell^{p}$. If $x_{n}\rightharpoonup x$, then the sequence
$(x_{n})$ is bounded and converges to $x$ by components. Then $(x_{n}-x)^{*}\rightharpoonup0$
in $\ell^{p'}$, and by Theorem \ref{thm:conj} it follows that $x_{n}\rightharpoondown x$. 

Conversely, if $x_{n}\rightharpoondown x$, then by Theorem \ref{thm:conj}
$(x_{n}-x)^{*}\rightharpoonup0$ in $\ell^{p'}$, and then $x_{n}$
converges to $x$ by components. Since by Theorem \ref{thm:newBS}
$\Delta$-convergent sequences are bounded, this implies that $x_{n}\rightharpoondown x$. \end{proof}
\begin{rem}
Another proof that weak and $\Delta$-convergence in Hilbert space
coincide can be infered from the definition of $\Delta$-convergence,
Proposition \ref{thm:3eqdef} and the elementary identity 
\[
\|x_{n}-x+y\|^{2}=\|x_{n}-x\|^{2}+\|y\|^{2}+2(x_{n}-x,y).
\]

\end{rem}

\begin{rem}
From Theorem \ref{thm:conj} it follows that $\Delta$-limit is not
additive, that is, the relation $\stackrel{\rightharpoondown}{\lim}(x_{n}+y_{n})=\stackrel{\rightharpoondown}{\lim}x_{n}+\stackrel{\rightharpoondown}{\lim}y_{n}$
is generally false. Consider, for example, $L^{4}((0,9))$. Set $x_{0}(t)=2$
for $t\in(0,1]$ and $x_{0}(t)=-1$ for $t\in(1,9]$. Define $x_{n}(t)=x_{0}(nt)$
when $\text{0<t\ensuremath{\le\frac{9}{n}}}$ and extend it periodically
to $(0,9)$. Set $y_{0}(t)=-1$ for $t\in(0,\frac{9}{2}]$ and $y_{0}(t)=1$
for $t\in(\frac{9}{2},9]$ and define $y_{n}$ similarly to $x_{n}$.
Observe that $x_{n}^{3}\rightharpoonup0,$ $y_{n}^{3}\rightharpoonup0$,
but $(x_{n}+y_{n})^{3}\rightharpoonup\frac{1}{2}$.
\end{rem}

\begin{rem}
Using Theorem \ref{thm:conj} one can also show that norms are not
necessarily lower semicontinuous with respect to $\Delta$-convergence.
Let $(v_{k})$ be a normalized sequence in $L^{4}([0,1])$, such that
$v_{k}\rightharpoondown0$ and $v_{k}\rightharpoonup a$ where $a$
is a positive constant (one constructs such sequence by fixing a step
function $v_{0}$ such that $\int v_{0}^{3}=0$ and $\int v_{0}>0$,
rescaling it by the factor $k$ and extending it periodically). By
Theorem \ref{thm:conj}, $v_{k}^{3}\rightharpoonup0$ in $L^{4/3}$.
Let $u_{k}=u-tv_{k}$, $t>0$ with some positive function $u$. Then
\[
\int u^{4}-\int u_{k}^{4}=4t^{3}\int uv_{k}^{3}-6t^{2}\int u^{2}v_{k}^{2}+4t\int u^{3}v_{k}-t^{4}\int v_{k}^{4}
\]
\[
\ge-6t^{2}\int u^{2}v_{k}^{2}+4ta\int u^{3}-t^{4}+to(1).
\]
Taking into account that $\int u^{2}v_{k}^{2}$ is bounded as $k\to\infty$,
we have that for $t$ sufficiently small the right hand side is bounded
away from zero for all $k$ sufficiently large.
\end{rem}

\subsection{$\Delta$-convergence versus weak convergence}

As we have shown above, $\Delta$-limits and weak limits coincide
in Hilbert spaces in $\ell^{p}$ -spaces, $1<p<\infty$ (Corollary
\ref{cor:lp}). In general it can happen that the weak limit and the
$\Delta$-limit of a sequence both exist but are different.
\begin{example}
\label{rem:bad}%
\begin{comment}
\begin{rem}
In the $L^{p}$-spaces, $p\neq2$, the metric weak limit and weak
limit might have different values. Let $\psi(x)=2$ for $x\in(0,1)$
and $\psi(x)=-1$ for $x\in(1,3)$. Let $\psi_{n}(x)=\psi(3nx)$ on
$x\in(0,1/n)$, extended periodically to $(0,1)$. Clearly (by using
smooth test functions) $\psi_{n}\rightharpoonup0$ in $L^{p}(0,1)$
for every $p\in(1,\infty)$. The metric weak limit of $\psi_{n}$,
as it follows immediately from Theorem \ref{thm:conj}, is a constant
that can be found by direct computation, as the value of $t\in[-1,2]$
for which $\int_{0}^{3}|\psi-t|^{p}=|(2-t)^{p}+2(t+1)^{p}$ is minimal,
namely, $t_{p}=2\frac{2-2^{1/(p-1)}}{1+2^{1/(p-1)}}$ and $t_{p}=0$
only if $p=2$. The two modes of weak convergence differ also in the
Sobolev spaces $W^{1,p}$, $p\neq2$: set $\Psi_{n}$ to be the primitive
function of $\psi_{n}$ normalized in $W^{1,p}((0,1))$. \end{rem}
\end{comment}
An example of Opial \cite[Section 5]{Opial} allows an immediate interpretation
in terms of $\Delta$-limit and then says that in the space $L^{p}((0,2\pi))$,
$p\neq2$, $1<p<\infty$, there exist sequences whose $\Delta$-limit
and weak limit are different functions (M. Cwikel has brought the
authors' attention to the fact that the number $3/4$ which appears
twice in the definition of function $\phi$ on p. 596 of \cite{Opial}
is a misprint and is to be read in both places as $4/3$).\end{example}
\begin{rem}
Furthermore, if $\Psi_{n}$ is the primitive function of $\psi_{n}$
of Opial's counterexample, normalized in $W^{1,p}((0,2\pi))$, the
sequence $\{\Psi_{n}\}$ in $W^{1,p}((0,2\pi))$ also has a $\Delta$-limit
and a weak limit with different values (note that because of the normalization
coefficient the non-gradient portion of the Sobolev norm for this
sequence is vanishing). 
\end{rem}

\begin{rem}
It is not clear at this point when $\Delta$-convergence can be associated
with a topology, except when $\Delta$-convergence coincides with
weak convergence. See a preliminary discussion in \cite{st3}.
\end{rem}

\begin{rem}
\label{rem: snl1}In general, weakly lower semicontinuous functionals
are not lower semicontinuous with respect to $\Delta$-convergence.
From Example \ref{rem:bad} it follows that this is the case already
for continuous linear functionals acting on $L^{p}$, $p\neq2$. 
\end{rem}

\subsection{The Opial's condition in uniformly convex spaces. }

In this subsection we show that the Opial's condition (Condition (2)
in \cite{Opial}), which plays significant role in the fixed point
theory, has, for uniformly convex and uniformly Banach spaces, two
equivalent formulations. One is that weak and $\Delta$-convergence
coincide and the other is that the Frechét derivative of the norm
is weak-to-weak continuous away from zero. The latter is similar to
Lemma 3 in \cite{Opial} (which makes a weaker assertion under weaker
conditions). 
\begin{defn}
\label{def:StrongOpial}Let $X$ be a Banach space. One says that
a sequence $(x_{n})_{n\in\N}\subset X$, which is weakly convergent
to a point $x_{0}\in X$, satisfies the \emph{Opial's condition} if
\begin{equation}
\liminf\|x_{n}-x_{0}\|\le\liminf\|x_{n}-x\|\mbox{ for every }x\in X.\label{eq:OpCon}
\end{equation}
One says that a Banach space $X$ satisfies the Opial's condition
if any weakly convergent sequence $(x_{k})_{k\in\mathbb{N}}$ in $X$
satisfies the Opial's condition.\end{defn}
\begin{rem}
It is immediate from respective definitions that if a sequence in
a Banach space satisfies the Opial's condition and is both weakly
convergent and $\Delta$-convergent, then its $\Delta$-limit equals
its weak limit.\end{rem}
\begin{thm}
\label{thm:UCOpial} Let $X$ be a uniformly convex and uniformly
smooth Banach space. Then $X$ satisfies the Opial's condition if
and only if for any sequence $(x_{n})_{n\in\N}\subset X$,

\begin{equation}
x_{n}\rightharpoonup x\Longleftrightarrow x_{n}\rightharpoondown x,\label{eq:updown}
\end{equation}
or, equivalently, if for any bounded sequence which does not have
a strongly convergent subsequence, 
\begin{equation}
x_{n}\rightharpoonup x\text{ in }X\;\Longleftrightarrow(x_{n}-x)^{*}\rightharpoonup0\text{ in }X^{*}.\label{eq:dualityOpial}
\end{equation}
\end{thm}
\begin{proof}
The Opial's condition follows immediately from (\ref{eq:updown})
and the definition of $\Delta$-convergence. Asume now that Opial's
condition holds. By the Banach-Alaoglu Theorem and Theorem \ref{thm:newBA}
(once we take into account Theorem \ref{thm:newBS}), it suffices
to consider sequences that have both a weak and a $\Delta$-limit.
Then by (\ref{eq:OpCon}) the weak limit of such sequence satisfes
the deinition of $\Delta$-limit. The last assertion of the theorem
follows from Theorem \ref{thm:conj}.
\end{proof}

\begin{rem}
It should also be noted that $\Delta$-convergence, unlike weak convergence,
depends on the choice of an equivalent norm. Theorem 1 of van Dulst
\cite{VD}, proves that in a separable Banach space one can always
find an equivalent norm (that one may call a \emph{van Dulst norm})
such that every weakly convergent sequence in the space satisfies
Opial's condition (\ref{eq:OpCon}), i.e. that $\Delta$-convergence
associated with a van Dulst norm is associated with the weak topology.
In practice, however, renorming the space may change conditions of
a problem where the Opial's condition is needed. In particular, since
van Dulst's construction uses a basis in a Banach space $Y$ which
contains $X$ isometrically,%
\begin{comment}
 ((05/12/13 It could happen (as shown by Enflo) that $X$ itself does
not have a basis.))
\end{comment}
{} it is not clear if one can preserve the invariance of the equivalent
norm with respect to a given group of operators without existence
of a wavelet basis associated with this group. Theorem \ref{thm:main-1}
requires that the new norm will remain uniformly convex and invariant
with respect to a fixed group of isometries, which is not assured
by the van Dulst's construction. For the purpose of applications to
functional spaces, uniformly convex norms, satisfying strong Opial's
condition and invariant with respect to Euclidean shifts and dyadic
dilations, are known (Cwikel \cite{Cwikel-PD}) for Besov and Triebel-Lizorkin
spaces $\dot{B}^{s,p,q}$ and $\dot{F}^{s,p,q}$ with $p,q\in(1,\infty)$,
$s\in\R$ (which includes Sobolev spaces $\dot{H}^{s,p}$ for all
$s\in\R$, $p\in(1,\infty)$)) for all Besov and Triebel-Lizorkin
spaces $\dot{B}^{s,p,q}$ and $\dot{F}^{s,p,q}$ with $p,q\in(1,\infty)$,
$s\in\R$ (which includes Sobolev spaces $\dot{H}^{s,p}(\R^{N})$
for all $s\in\R$ and $p\in(1,\infty)$). Motivation for the choice
of norm, based on the Littlewood-Paley decomposition, can be found
the proof of cocompactness of Sobolev imbeddings in Killip \&Visan,
\cite{KiVi}, Chapter 4 (note that the authors call the property of
cocompactness\emph{ inverse imbedding}). The argument of Cwikel is
based on verifying \ref{eq:dualityOpial} using the definition of
the equivalent norm for Besov and Triebel-Lizorkin spaces from \cite{Triebel}
(Definition 2, p.~238), based on the Littlewood-Paley decomposition,
and it reduces both weak and polar convergence, by straightforward
calculations, to obvious pointwise convergence of the sequence $(2^{ns}F^{-1}\varphi_{0}(2^{-n}\cdot)F(u_{k}-u))_{k\in\N}$,
where $F$ denotes the Fourier transform, $\varphi_{0}$ is a smooth
function supported in an annulus, $n\in\Z$ and $s\in\R$. 
\end{rem}

\section{A discussion concerning the Brezis-Lieb lemma }

It is interesting to note that while weak convergence of $(x_{k})_{k\in\mathbb{N}}$
to an element $x$ in a Banach space implies that $\|x_{k}\|\ge\|x\|+o(1)$
(weak lower semicontinuity of the norm), $\Delta$-convergence of
such a sequence to $x$ implies that $\|x_{k}\|\ge\|x_{k}-x\|+o(1)$,
while in the case of sequences in a Hilbert space, both of these inequalities
can also be deduced from the stronger condition 
\begin{equation}
\|x_{k}\|^{2}=\|x_{k}-x\|^{2}+\|x\|^{2}+o(1)\label{eq:HilbBl}
\end{equation}
When the space $X$ is uniformly convex, Lemma \ref{lem:energy} gives
a lower bound for the norm of the $\Delta$-convergent sequence in
the form $\|u_{k}\|\ge\|u_{k}-u\|+\delta(\|u\|)+o(1)$. Another relation
that allows to estimate the norm of the sequence $(u_{k})$ by the
norms of its weak limit $u$ and of the remainder sequence $u_{k}-u$
when $X=L^{p}$, $1\le p<\infty$, is the important Brezis-Lieb lemma
\cite{Brezis-Lieb}. Remarkably, in the case $p=2$ Brezis-Lieb lemma
follows from (\ref{eq:HilbBl}), while for $p\neq2$ it requires,
in addition to the assumption of weak convergence also convergence
almost everywhere. One may, however, interpret convergence a.e. as
a sufficient condition for $\Delta$-convergence of the sequence to
its weak limit, as one can see from the Brezis-Lieb lemma itself,
or, alternatively, from the following argument.
\begin{lem}
Let $(\Omega,\mu)$ be a measure space and let $u_{k}$ be a bounded
sequence in $L^{p}(\Omega,\mu$), $p\in(1,\infty)$. If $u_{k}\rightharpoonup u$
and $u_{k}\to u$ a.e. then $u_{k}\rightharpoondown u$. \end{lem}
\begin{proof}
Without loss of generality we may assume that $u=0$. Let $u_{k}^{*}\rightharpoondown w$
on a renamed subsequence. Then $w=0$ on every set where a.e. convergence
becomes uniform, and therefore, by Egoroff theorem, $w=0$ outside
of a set of arbitrarily small measure, and thus a.e. Thus $u_{k}^{*}$
has no subsequence with a non-zero $\Delta$-limit, i.e. $u_{k}\rightharpoondown0$.
\end{proof}
It is natural to pose the question, what may remain of the assertion
of the Brezis-Lieb lemma if we replace its conditions with a weaker
requirement that both $\Delta$-limit and weak limit exist and are
equal. We have
\begin{thm}
\label{thm:newbl}Let $(\Omega,\mu)$ be a measure space. Assume that
$u_{k}\rightharpoonup u$ and $u_{k}\rightharpoondown u$ in $L^{p}(\Omega,\mu)$.
If $p\ge3$ then

\begin{equation}
\int_{\Omega}|u_{k}|^{p}d\mu\ge\int_{\Omega}|u|^{p}d\mu+\int_{\Omega}|u_{k}-u|^{p}d\mu+o(1).\label{eq:BL}
\end{equation}
\end{thm}
\begin{proof}
In order to prove the assertion it suffices to verify the elementary
inequality 

\begin{equation}
(1+t)^{p}\ge1+|t|^{p}+p|t|^{p-2}t+pt,\label{eq:elementary}
\end{equation}
since it implies $|u_{k}|^{p}\ge|u_{k}-u|^{p}+|u|^{p}+p|u|^{p-2}u(u_{k}-u)+p|u_{k}-u|^{p-2}(u_{k}-u)u$,
with the integrals of the last two terms vanishing by assumption.
The elementary inequality is equivalent to the inequalities
\[
f_{+}(t)=(1+t)^{p}-1-t^{p}-pt^{p-1}-pt\ge0,\quad t\geq0
\]
and, assuming without any restriction (in view of the symmetry of
the formula) that $|t|\leq1$ 
\[
f_{-}(t)=(1-t)^{p}-1-t^{p}+pt^{p-1}+pt\ge0,\quad t\in[0,1].
\]
To prove them, note that both functions vanish at zero, so it suffices
to show that their derivatives are nonnegative. We have
\[
\frac{1}{p}f'_{+}(t)=(1+t)^{p-1}-t^{p-1}-1-(p-1)t^{p-2},
\]
which is also a function vanishing at zero, so it suffices to show
that its derivative is nonnegative, i.e. 
\[
\frac{1}{p(p-1)}f''_{+}(t)=(1+t)^{p-2}-t^{p-2}-(p-2)t^{p-3}\ge0.
\]
Let $s=t^{-1}$ and $q=p-2$. Then 
\[
\frac{s^{q}}{p(p-1)}f''_{+}(s^{-1})=(1+s)^{q}-1-qs\ge0,\quad s\ge1,
\]
which is true by convexity of the first term, since $q\ge1$ (i.e.
$p\ge3$). 

Consider now the derivative of $f_{-}$:
\[
\frac{1}{p}f'_{-}(t)=-(1-t)^{p-1}-t^{p-1}+1+(p-1)t^{p-2}.
\]
It remains to notice that $(1-t)^{p-1}+t^{p-1}\le1$.\end{proof}
\begin{rem}
Easy calculations show that inequality (\ref{eq:elementary}) used
in the proof of Theorem \ref{thm:newbl} does not hold unless $p\ge3$,
and the argument of homogenization type is used in \cite{AditiBL}
to show that condition $p\ge3$ is indeed necessary for (\ref{eq:BL}),
unless $p=2$. For $p=2$, as we already mentioned, inequality (\ref{eq:BL})
holds, and, moreover, becomes an equality, which can be easily verified.
\end{rem}

\begin{rem}
The inequality in (\ref{eq:BL}) can be strict. Indeed, one can easily
calculate by binomial expansion for $p=4$ that if $u_{k}\rightharpoonup u$
and $u_{k}\rightharpoondown u$ ( i.e. $(u_{k}-u)^{3}\rightharpoonup0$
in $L^{4/3})$, then
\end{rem}
\[
\int_{\Omega}|u_{k}|^{4}d\mu=\int_{\Omega}|u|^{4}d\mu+\int_{\Omega}|u_{k}-u|^{4}d\mu+6\int u^{2}(u_{k}-u)^{2}d\mu+o(1).
\]
Let $\Omega=(0,3)$ equipped with Lebesgue measure and consider three
sequences of disjoint sets $A_{1;k},\dots,A_{3;k}$ , $k\in\N$, such
that $(\frac{m-1}{k},\frac{m}{k})\subset A_{\mathrm{rem}(m,3);k}$
where $\mathrm{rem}(m,3)$ is the remainder of division of $m$ by
3 and $m=1,\dots,3k$. Set $u_{k}=\sum_{i=1}^{3}a_{i}\chi_{A_{i;k}}$
where $a_{1}=1,$ $a_{2}=2$ and $a_{3}=0$. Then $u_{k}\rightharpoonup1$
and $(u_{k}-u)^{3}\rightharpoonup\frac{1}{3}\sum_{i}(a_{i}-1)^{3}=0$,
while $\int u^{2}(u_{k}-u)^{2}d\mu\to2>0$.

\begin{rem}
$\Delta$-convergence is\emph{ necessary }for the assertion of Brezis-Lieb
lemma, and even a weaker statement (\ref{eq:BL}), to hold. More accurately,
if a sequence $(u_{k})\subset L^{p}(\Omega,\mu)$, $p\in[1,\infty)$,
and a function $u\in L^{p}(\Omega,\mu)$ are such that for any $v\in L^{p}(\Omega,\mu)$,
\begin{equation}
\int_{\Omega}|u_{k}-v|^{p}d\mu\ge\int_{\Omega}|u-v|^{p}d\mu+\int_{\Omega}|u_{k}-u|^{p}d\mu+o(1),\label{eq:BL-1}
\end{equation}
 then $u_{k}\rightharpoondown u$ by the definition of $\Delta$-limit.
\end{rem}

\section{profile decomposition in terms of $\Delta$-convergence}

Throughout this section we assume that $X$ is a uniformly convex
and uniformly smooth Banach space. We also assume that $D$ is a subset,
containing the identity operator, of a group $D_{0}$ of isometries
on $X$. In this section we prove that every bounded sequence in $X$
has a subsequence with a profile decomposition based on $\Delta$-convergene.The
reason that motivates us to define concentration profiles as $\Delta$-limits,
rather than weak limits, is that $\Delta$-convergence yields estimates
of the energy type (\ref{eq:energy}) which are not readily available
when usual weak convergence is used. 

We need to modify some of the definitions of previous sections, which
are based on weak convergence, by changing the mode of convergence
involved to $\Delta$-convergence.
\begin{defn}
One says that a sequence $(u_{k})\subset X$ has a $D$-$\Delta$-limit
$u$ (to be denoted $u_{k}\stackrel{D}{\rightharpoondown}u)$, if
for every sequence $(g_{k})\subset D$, $g_{k}^{-1}(u_{k}-u)\rightharpoondown0$.
\end{defn}
Equivalently, if we take into account the supremum of the norms of
the $\Delta$-profiles of a given sequence by setting

\[
p((u_{k})_{k\in\N})=\sup\{\|w\|:\;\exists\text{ subsequences }(u_{n_{k}})\subset(u_{k})\text{ and }(g_{k})\subset D,
\]
\[
\text{ \text{ such that }}g_{k}^{-1}(u_{n_{k}})\rightharpoondown w\},
\]
we can say that $u_{k}\stackrel{D}{\rightharpoondown}u$ if and only
if $p((u_{k}-u)_{k\in\N})=0$.

\begin{defn}
\label{def:PPD}One says that a bounded sequence $(u_{k})$ in a Banach
space $X$\textbf{ }\textit{admits a }$\Delta$-\textit{profile decomposition
relative to the set of isometries} $D\subset D_{0}$, if there exist
sequences $(g_{k}^{(n)})_{k}\subset D$ with $g_{k}^{(1)}=\mathrm{Id}$,
elements $w^{(n)}\in X$, $n\in\mathbb{N}$, and a sequence $r_{k}\stackrel{D}{\rightharpoondown}0$
such that 
\begin{equation}
(g_{k}^{(n)})^{-1}g_{k}^{(m)}\rightharpoonup0\text{ whenever }m\neq n\text{ (asymptotic decoupling of gauges)},\label{eq:decouple-1}
\end{equation}
and a renamed subsequence of $u_{k}$ can be represented in the form
\begin{equation}
u_{k}=\sum_{j=1}^{\infty}g_{k}^{(j)}w^{(j)}+r_{k},\label{eq:ProDec-1}
\end{equation}
where the series $\sum_{j=1}^{\infty}g_{k}^{(j)}w^{(j)}$ is convergent
in $X$ absolutely and uniformly with respect to $k$. In this case
we also have 
\[
(g_{k}^{(n)})^{-1}u_{k}\rightharpoondown w^{(n)},\; n\in\mathbb{N}.
\]

\end{defn}

\begin{defn}
We shall say that the group $D_{0}$ of isometries on a Banach spac
$X$ is a \emph{dislocation group} (to be denoted $D_{0}\in\mathcal{I}_{X}$)
if it satisfies

\begin{equation}
(g_{k})\subset D_{0},\; g_{k}\not\rightharpoonup0\Longrightarrow\exists(k_{j})\subset\mathbb{N}:\;(g_{k_{j}}^{-1})\text{ and (\ensuremath{g_{k_{j}}}) converge operator-strongly (i.e. pointwise),}\label{eq:*}
\end{equation}

and 
\begin{equation}
u_{k}\rightharpoondown0,\; w\in X,\;(g_{k})\subset D_{0},\, g_{k}\rightharpoonup0\;\Longrightarrow\; u_{k}+g_{k}w\rightharpoondown0.\label{eq:**}
\end{equation}
\end{defn}
\begin{rem}
Note also that condition (\ref{eq:**} ) is trivially satisfied if
Opial's condition holds, and in particular, in a Hilbert space, so
this definition agrees with the definition of the dislocation group
used previously in \cite{Fieseler-Tintarev}. It is easy to prove
that when $D_{0}$ is a dislocation group, the profiles $w^{(n)}$
in Definition \ref{def:PPD} are unique, up to the choice of subsequence
and up to multiplication by an operator $g\in D_{0}$. The argument
is repetitive of that in Proposition 3.4 in \cite{Fieseler-Tintarev},
which considers the case of Hilbert space.\end{rem}
\begin{thm}
\label{thm:main}Let $X$ be a uniformly convex and uniformly smooth
Banach space and let $D\ni\mathrm{Id}$ be subset of a dislocation
group $D_{0}$. Then every bounded sequence $(x_{k})\subset X$ admits
a $\Delta$-profile decomposition relative to $D$. Moreover, if $\|x_{k}\|\le1$,
and $\delta$ is the modulus of convexity of $X$, then $\|w^{(n)}\|\le2$
for all $n\in\N$ and
\begin{equation}
\limsup\|r_{k}\|+\sum_{n}\delta(\|w^{(n)}\|)\le1.\label{eq:Energy}
\end{equation}

\end{thm}
We prove the theorem via a sequence of lemmas. 
\begin{lem}
\label{lem:w-1}Let $(g_{k})\subset D_{0}$. If $g_{k}\rightharpoonup0$
then $g_{k}^{-1}\rightharpoonup0$.\end{lem}
\begin{proof}
Assume that $g_{k}^{-1}\not\rightharpoonup0$. Then by (\ref{eq:*})
the sequence $(g_{k})$ has a strongly convergent subsequence, whose
limit is an isometry, and thus it cannot be zero.
\end{proof}

\begin{lem}
\label{lem:weakstrong}Let $(g_{k})\subset D$ be such that $g_{k}^{-1}$
is operator-strongly convergent. If $x_{k}\rightharpoondown0$, then
$g_{k}x_{k}\rightharpoondown0$.\end{lem}
\begin{proof}
It is immediate from the assumption that there is a linear isometry
$h$, such that $g_{k}^{-1}y\to hy$ for every $y\in X$. Then

\[
\langle(g_{k}x_{k})^{*},y\rangle=\langle x_{k}^{*},g_{k}^{-1}y\rangle=\langle x_{k}^{*},hy\rangle+o(1)\to0.
\]

\end{proof}
Our next lemma assures that dislocat\i on sequences $(g_{k})$ that
provide distinct profiles are asymptotically decoupled.
\begin{lem}
\label{lem:1-2}Let $(u_{k})\subset X$ be a bounded sequence. Assume
that there exist two sequences $(g_{k}^{(1)})_{k}\subset D$ and \textup{$(g_{k}^{(2)}){}_{k}\subset D$},
such that $(g_{k}^{(1)})^{-1}u_{k}\rightharpoondown w^{(1)}$ and
$(g_{k}^{(2)})^{-1}(u_{k}-g_{k}^{(1)}w^{(1)})\rightharpoondown w^{(2)}\neq0$.
Then $(g_{k}^{(1)})^{-1}(g_{k}^{(2)})\rightharpoonup0$.\end{lem}
\begin{proof}
Assume that $(g_{k}^{(1)})^{-1}(g_{k}^{(2)})$ does not converge weakly
to zero. Then by (\ref{eq:*}), on a renamed subsequence, $(g_{k}^{(1)})^{-1}(g_{k}^{(2)})$
converges operator-strongly to some isometry $h$. Then by Lemma \ref{lem:weakstrong},
\[
(g_{k}^{(1)})^{-1}(g_{k}^{(2)})[(g_{k}^{(2)})^{-1}(u_{k}-g_{k}^{(1)}w^{(1)})-w^{(2)}]\rightharpoondown0,
\]
which implies, taking into account (\ref{eq:**}), 
\[
(g_{k}^{(1)})^{-1}u_{k}-w^{(1)}-hw^{(2)}\rightharpoondown0.
\]
However, this contradicts the definition of $w^{(1)}$ and the assumption
that $w^{(2)}\ne0$.
\end{proof}
The next statement assures that one can find decoupled elementary
concentrations by iteration.
\begin{lem}
\label{lem:induction}Let $u_{k}$ be a bounded sequence in $X$ and
let $(g_{k}^{(n)})_{k}\subset D$, $w^{(n)}\in X$, $n=1,\dots,M$,
be such that $g_{k}^{(1)}=I$, $(g_{k}^{(n)})^{-1}u_{k}\rightharpoondown w^{(n)}$,
$n=1,\dots M$, and \textup{$(g_{k}^{(n)})^{-1}(g_{k}^{(m)})\rightharpoonup0$}
whenever $n<m$. Assume that there exists a sequence $(g_{k}^{(M+1)})\subset D$
such that, on a renumbered subsequence, \textup{$(g_{k}^{(M+1)})^{-1}(u_{k}-w^{(1)}-g_{k}^{(2)}w^{(2)}-\dots-g_{k}^{(M)}w^{(M)})\rightharpoondown w^{(M+1)}\neq0$.
}Then $(g_{k}^{(n)})^{-1}(g_{k}^{(M+1)})\rightharpoonup0$ for $n=1,\dots M$.\end{lem}
\begin{proof}
We can replace $u_{k}$ by $u_{k}-\sum_{m\neq n}g_{k}^{(m)}w^{(m)}$
and then, thanks to (\ref{eq:**}), apply Lemma \ref{lem:1-2} with
$1$ replaced by $n$ and $2$ by $M+1$. 
\end{proof}
We may now start the construction needed for the proof of Theorem
\ref{thm:main}. As we have remarked before, we may without loss of
generality assume that $\|x_{k}\|\le1$. 

Let us introduce a partial strict order relation between sequences
in $X$, to be denoted as $>$. First, given two sequences $(x_{k})\subset X$
and $(y_{k})\subset X$, we shall say that $(x_{k})\succ(y_{k})$
if there exists a sequence $(g_{k})\subset D$, an element $w\in X\setminus\{0\}$,
and a renumeration $(n_{k})$ such that $g_{n_{k}}^{-1}x_{n_{k}}\rightharpoondown w$
and $y_{k}=x_{n_{k}}-g_{n_{k}}w$$ $. From Lemma \ref{lem:energy}
it follows that if $(x_{k})\succ(y_{k})$ and $\|x_{k}\|\le1$, then
$\|y_{k}\|\le1$ for $k$ sufficiently large, and therefore it follows
from sequential $\Delta$-compactness of bounded sequences that for
every sequence $(x_{k})\subset X$, $\|x_{k}\|\le1$, which is not
$D$-$\Delta$- convergent to $0$, there is a sequence $(y_{k})\subset X$,
such that $\|y_{k}\|\le1$ and $(x_{k})\succ(y_{k})$. 

Then we shall say that $(x_{k})>(y_{k})$ in one step, if $(x_{k})\succ(y_{k})$
and in $m$ steps, $m\ge2$, if there exist sequences $(x_{k}^{1})\succ(x_{k}^{2})\succ\dots\succ(x_{k}^{m})$,
such that $(x_{k}^{1})=(x_{k})$ and $(x_{k}^{m})=(y_{k})$. Note
that, for every sequence $(x_{k})\subset X$, $\|x_{k}\|\le1$, either
there exists a finite number of steps $m_{0}\in\N$ such that $(x_{k})>(y_{k})$
in $m_{0}$ steps for some $(y_{k})\subset X$, $\|y_{k}\|\le1$,
and $p((y_{k}))=0$, or for every $m\in\N$ there exists a sequence
$(y_{k})\subset X$, $\|y_{k}\|\le1$, such that $(x_{k})>(y_{k})$
in $m$ steps. We will say that $(x_{k})\ge(y_{k})$ if either $(x_{k})>(y_{k})$
or $(x_{k})=(y_{k})$

Define now 
\[
\sigma((x_{k}))=\inf_{(y_{k})\ge(x_{k})}\sup_{k}\|y_{k}\|
\]
 and observe that if $(x_{k})\ge(z_{k})$, then $\sigma((x_{k}))\le\sigma((z_{k}))$,
since the set of sequences $(y_{k})$ dominating $(z_{k})$ is a subset
of sequences dominating $(x_{k})$. 
\begin{lem}
\textup{\label{lem:tail}Let $(x_{k})>(y_{k})$ in $m$ steps, $\|x_{k}\|\le1$
and $\eta>0$. Then there exist elements $w^{(1)}$,$\dots$,$w^{(m)}$,
and sequences $(g_{k}^{(1)})$,$\dots$, $(g_{k}^{(m)})$ in $D$,
and a renumeration $(n_{k})$ such that}
\[
y_{k}=x_{n_{k}}-\sum_{n=1}^{m}g_{n_{k}}^{(n)}w^{(n)},
\]
\textup{ $(g_{n_{k}}^{(p)})^{-1}g_{n_{k}}^{(q)}\rightharpoonup0$
for $p\neq q$, and for any set $J\subset J_{m}=(1,\dots,m)$, }

\begin{equation}
\delta(\sum_{n\in J}g_{n_{k}}^{(n)}w^{(n)})\le\sup\|x_{n_{k}}\|-\sigma((x_{n_{k}}))+\eta,\;\mbox{ for all }k\mbox{ sufficiently large}.\label{eq:tail}
\end{equation}
\end{lem}
\begin{proof}
The first assertion follows from Lemma \ref{lem:induction}. Let 
\[
\alpha_{k}=x_{n_{k}}-\sum_{n\in J_{m}\setminus J}g_{k}^{(n)}w^{(n)},
\]
\[
\beta_{k}=x_{n_{k}}-\sum_{n\in J_{m}\setminus J}g_{k}^{(n)}w^{(n)}-\frac{1}{2}\sum_{n\in J}g_{k}^{(n)}w^{(n)}=\frac{1}{2}(\alpha_{k}+y_{k}).
\]
By Lemma \ref{lem:energy}, $\|y_{k}\|\le\|\alpha_{k}\|\le\|x_{k}\|\le1$
and $\beta_{k}\le1$ for all $k$ large. Note that, as in the construction
above, we can take $k$ large enough so that $\sup\|\beta_{k}\|\le\inf\|\beta_{k}\|+\eta$.
By uniform convexity, for large $k$ we have
\[
\|\beta_{k}\|\le\|\alpha_{k}\|-\delta(\alpha_{k}-y_{k}).
\]
Therefore
\[
\delta(\|\sum_{n\in J}g_{k}^{(n)}w^{(n)}\|)\le\|\alpha_{k}\|-\|\beta_{k}\|\le\sup\|x_{k}\|-\sigma((x_{k}))+\eta.
\]

\end{proof}
\emph{Proof of Theorem \ref{thm:main}.} For every $j\in\N$ define
$\epsilon_{j}=\delta(\frac{1}{2^{j}})$. Let $(x_{k}^{(1)})\subset X$
be such that $(x_{k})>(x_{k}^{(1)})$ and $\mbox{\ensuremath{\sup}}\|x_{k}^{(1)}\|<\sigma((x_{k}))+\epsilon_{1}$.
Consider the following iterations. Given $(x_{k}^{(j)})_{k}$, either
$p((x_{k}^{(j)})_{k})=0$, in which case there is a profile decomposition
with $r_{k}=x_{k}^{(j)}$, or there exists a sequence $(x_{k}^{(j+1)})_{k}<(x_{k}^{(j)})_{k}$,
such that $\mbox{\ensuremath{\sup}}_{k}\|x_{k}^{(j+1)}\|<\sigma((x_{k}^{(j)})_{k})+\frac{\epsilon_{j}}{2}$,
$j\in\N$. Let us denote as $n_{k}^{j}$ the cumulative enumeration
of the original sequence that arises at the $j$-th iterative step,
and denote as $m_{j+1}$ the number of elementary concentrations that
are subtracted at the transition from $(x_{k}^{(j)})_{k}$ to $(x_{k}^{(j+1)})_{k}$
(using the convention $x_{k}^{(0)}:=x_{k}$). Set $M_{j}=\sum_{i=1}^{j}m_{i}$,
$M_{0}=0$. Then the sequence $(x_{k}^{(j)})_{k}$ admits the following
representation:
\[
x_{k}^{(j)}=x_{n_{k}^{j}}-\sum_{n=1}^{M_{j}}g_{n_{k}^{j}}^{(n)}w^{(n)},\; k\in\N.
\]
By Lemma \ref{lem:tail}, under an appropriate renumeration such that
(\ref{eq:tail}) holds for all $k$, 
\[
\delta(\|\sum_{n=M_{j-1}+1}^{M_{j}}g_{n_{k}^{j}}^{(n)}w^{(n)}\|)\le\sup\|x_{k}^{(j+1)}\|-\sigma((x_{k}^{(j)}))+\frac{\epsilon_{j}}{2}<\epsilon_{j},\; k\in\N,
\]
and thus 
\[
\|\sum_{n=M_{j-1}+1}^{M_{j}}g_{n_{k}^{j}}^{(n)}w^{(n)}\|\le2^{-j},\; j\in\N.
\]
Let us now diagonalize the double sequence $x_{k}^{(j)}$ by considering
\[
x_{k}^{(k)}=x_{n_{k}^{k}}-\sum_{n=1}^{M_{k}}g_{n_{k}^{k}}^{(n)}w^{(n)}.
\]

Let us show that $x_{k}^{(k)}\stackrel{D}{\rightharpoondown}0$. Indeed,
by definition of functional $p$ and Lemma \ref{lem:tail}, $\delta(p(x_{k})\le\sup\|x\_k\|-\sigma(x_{k})$,
and therefore, for any $j\in\N$ and all $k\ge j$,
\[
p(x_{k}^{(k)})\le p(x_{k}^{(j)})\le\sup\|x_{k}^{(j)}\|-\sigma(x_{k}^{(j)})\le\epsilon_{j}.
\]
Since $j$ is arbitrary, this implies $p(x_{k}^{(k)})=0$. Furthermore,
denoting as $J_{j}$ an arbitrary subset, of $\{M_{j}+1,\dots,M_{j+1}\}$,
$j\in\N$, we have 
\[
\|\sum_{n=M_{k}+1}^{\infty}g_{n_{k}^{k}}^{(n)}w^{(n)}\|\le\sum_{j=k}^{\infty}\|\sum_{n\in J_{j}}g_{n_{k}^{k}}^{(n)}w^{(n)}\|\le\frac{1}{2^{k-1}}.
\]
We have therefore 
\[
x_{n_{k}^{k}}-\sum_{n=1}^{\infty}g_{n_{k}^{k}}^{(n)}w^{(n)}\stackrel{D}{\rightharpoondown}0,
\]
where the series is understood as the sum $S_{k}+S_{k}'$, where $S_{k}=\sum_{n=1}^{M_{k}}g_{n_{k}^{k}}^{(n)}w^{(n)}$
is a finite, not a priori bounded, sum, and a series $S'_{k}=\sum_{n=M_{k}+1}^{\infty}g_{n_{k}^{k}}^{(n)}w^{(n)}$
that converges unconditionally and uniformly in $k$.

Note, however, that $S_{k}$ is a sum of a bounded sequence $x_{n_{k}^{k}}$,
a $D$-$\Delta$-vanishing (and thus bounded) sequence, and the convergent
series $S'_{k}$ bounded with respect to $k$. Therefore the sum $S'_{k}$
is bounded with respect to $k$ and, consequently, the series $S_{k}+S_{k}'$
is convergent in norm, unconditionally and uniformly in $k$. Note
that the construction can be carried out without further modifications
if one prescribes in the beginning $g_{k}^{(1)}=\mathrm{Id}$ whenever
$w^{(1)}=\stackrel{\rightharpoondown}{\lim}x_{n_{k}}\neq0$, while
in the case $x_{k}\rightharpoondown0$ one can add the zero term $g_{k}^{(1)}w^{(1)}$
to the sum.\hfill{}$\Box$

\section{general properties of cocompactness and profile decompositions}

In this section we discuss some general functional-analytic properties
of sequences related to cocompactness, following the discussion for
Sobolev spaces in \cite{Solimini}. The reader whose interest is focused
on profile decompositions may skip to the next section after reading
the definition below. We will assume throughout this section that
$X$ is a strictly convex Banach space, unless specifically stated
otherwise, and that the set $D$ will be a non-empty subset of a group
$D_{0}$ of linear isometries on $X$. 
\begin{defn}
A continuous imbedding of two Banach spaces $X\hookrightarrow Y$,
given a set $D$ of bijective linear isometries of both $X$ and $Y$,
is called \emph{$D,X$-cocompact} (to be denoted $X\stackrel{D,X}{\hookrightarrow}Y$),
if any $D$-$\Delta$-convergent sequence in $X$ is convergent in
the norm of $Y$. It will be called \emph{$D,Y$-cocompact} (to be
denoted $X\stackrel{D,Y}{\hookrightarrow}Y$), if any sequence bounded
in $X$ and $D$-$\Delta$-convergent in $Y$, is convergent in the
norm of $Y$. 
\end{defn}
Note that when for weak and $\Delta$-convergence in $X$ (resp. $Y$)
coincide, $D,X$- (resp. $D,Y$-) cocompacntess coincides with $D$-cocompactness.

Analogously to the notion of $D$-cocompact set in Definition \ref{(Cocompact set)},
we say that a set $B\subset X$, is \emph{$D,X$-cocompact}, if every
$D$-$\Delta$-convergent sequence in $X$ is strongly convergent.
\begin{defn}
A subset $B$ of a Banach space $X$ is called \emph{$D$-}$\Delta$-\emph{bounded}
if for every sequence $(g_{k})\subset\mbox{D}$, $g_{k}\rightharpoonup0$,
and any sequence B, $x_{k}\rightharpoondown x$, one has $g_{k}^{-1}(x_{k}-x)\rightharpoondown0$.
It is called \emph{$D$-bounded} if it possesses analogous property
with $\Delta$-convergence replaced by weak convergence. 
\end{defn}

\begin{defn}
A Banach space $X$ is called locally D-$\Delta$-cocompact if every
bounded subset of $X$ is D-$\Delta$-cocompact. It is called locally
D-cocompact if it possesses analogous property with $\Delta$-convergence
replaced by weak convergence. 
\end{defn}
We we have two examples of locally cocompact spaces. 
\begin{example}
\label{ex:linfty}(cf. \emph{Remarks} on p. 395, \cite{Jaffard}).
The imbedding $\ell^{p}(\mathbb{Z})\hookrightarrow\ell^{\infty}(\mathbb{Z})$,
$1\le p\le\infty$, is $D$-cocompact with $D=\{u\mapsto u(\cdot+y)\}_{y\in\mathbb{Z}}$.
In particular, $\ell^{\infty}$ is locally cocompact. To see that
observe that $u_{k}\stackrel{D}{\rightharpoonup}0$ implies $u_{k}(y_{k})\to0$
for any $y_{k}$, in particular when $y_{k}$ is a point such that
$|u_{k}(y_{k})|\ge\frac{1}{2}\|u_{k}\|_{\infty}$. As an immediate
consequence we also have $\ell^{p}\stackrel{D}{\hookrightarrow}\ell^{q}$
whenever $q>p$.
\end{example}

\begin{example}
\label{ex:Linfty}Another example of a locally cocompact space is
$L^{\infty}(\mathbb{R})$, equipped with $D=\{u\mapsto u(2^{j}\cdot+y)\}_{j\in\mathbb{Z},y\in\mathbb{R}}$.
Indeed, assume, without loss of generality, that $A\ge\mathrm{ess}\sup u_{k}(x)=\|u_{k}\|_{\infty}\ge\eta>0$.
Then for every $k$ there exists a Lebesgue point $x_{k}$ of the
set $X_{k}=\{x:\, u_{k}(x_{k})\ge\eta/2\}$. Therefore, for every
$k$ and for every $\alpha\in(0,1)$ there exists $\delta_{\alpha,k}>0$
such that $|X_{k}\cap[x_{k}-\delta_{\alpha,k},x_{k}+\delta_{\alpha,k}]|\ge2\alpha\delta_{\alpha,k}$. 

Let and let $\tilde{u}_{k}(x)=u_{k}(\delta_{\alpha_{k}}^{-1}(x+x_{k}))$.
Consider the set $\tilde{X}_{k}=\{x:\,\tilde{u}_{k}(x_{k})\ge\eta/2\}$
and note that $|\tilde{X}_{k}\cap[-1,1]|\ge2\alpha$. Therefore, choosing
any $\alpha\in(\frac{2A}{2A+\eta},1)$, we get 
\[
\]
\[
\int_{[-1,1]}\tilde{u}_{k}\ge\alpha\eta-A(2-2\alpha)=(\eta+2A)\alpha-2A>0.
\]
Consequently, $\tilde{u}_{k}\not\rightharpoonup0$. It is easy to
show that if $j(\alpha,k)\in\mathbb{N}$ is such that $2^{-j(\alpha,k)}\le\delta_{\alpha,k}\le2^{1-j(\alpha,k)}$,
then $u_{k}(2^{j(\alpha,k)}(x+x_{k}))\not\rightharpoonup0$ as well
and $D$-cocompactness of bounded sets in $L^{\infty}(\mathbb{R})$
follows.
\end{example}
We have the following immediate criterion of local cocompactness. 
\begin{prop}
A Banach space $X$ is locally $D$-$\Delta$-cocompact ($D$-cocompact)
if and only if its every $D$-$\Delta$-bounded ($D$-bounded) set
is compact.\end{prop}
\begin{defn}
A set $B$ in a Banach space $X$ is called profile-compact relative
to a set $D$ of bijective isometries on $X$ if any sequence in $B$
admits a strong profile decomposition, i.e. a profile decomposition
whose remainder term vanishes in the norm of $X$. \end{defn}
\begin{prop}
Let $B$ be a profile-compact subset of a Banach space $X$ and let
$D$ be a nonempty subset of a dislocation group $D_{0}$. Then the
profiles $w^{(n)}$ for a sequence $(u_{k})\subset B$ are given by
$w^{(n)}=\stackrel{\rightharpoondown}{\lim}(g_{k}^{(n)})^{-1}u_{k}$
= $\stackrel{\rightharpoonup}{\lim}(g_{k}^{(n)})^{-1}u_{k}$.\end{prop}
\begin{proof}
Without loss of generality we may consider profile decompositions
with finitely many terms. Then from (\ref{eq:**}) by an elementary
induction argument, we see that the weak and the $\Delta$-limits
of $(g_{k}^{(n)})^{-1}u_{k}$ coincide.\end{proof}
\begin{rem}
Consider for simplicity a uniformly convex and uniformly smooth Banach
space $X$ with the Opial's condition. Conclusion of Theorem \ref{thm:main}
is analogous to the conclusion of the Banach-Alaoglu theorem, in the
sense that every bounded sequence is ``profile-weakly-compact''
(that is, has a subsequence that admits a profile decomposition).
Similarly to compactness of imbeddings, an imbedding $X\hookrightarrow Y$
is cocompact relative to a set of bijective isomertires $D\subset D_{0}\in\mathcal{I}_{X}$,
which extend to bijective isometries $D\subset D_{0}\in\mathcal{I}_{Y}$,
if and only if a ``profile-weakly-compact'' sequence in $X$ becomes
profile-compact (i.e. ``profile-strongly-compact'') in $Y$, that
is, if it gets a strongly vanishing remainder.
\end{rem}
Our next question is if a dual imbedding $Y^{*}\hookrightarrow X^{*}$
of a cocompact imbedding $X\hookrightarrow Y$ is cocompact. The answer
is positive, but it involves the two different modes of cocompactness
(or requires the Opial condition). 
\begin{prop}
\label{prop:dual coco}Let $X$ be a reflexive Banach space equipped
with a set $D$ of linear bijective isometries on $X$ and $Y$. Assume
that $X\stackrel{D,X}{\hookrightarrow}Y$, and that every bounded
sequence in $X$ admits a $\Delta$-profile decomposition. Then the
dual imbedding $Y^{*}\hookrightarrow X^{*}$ is $D^{\#}$-cocompact,
where 

\[
D^{\#}=\{(g^{*})^{-1},\; g\in D\}.
\]
\end{prop}
\begin{proof}
Consider $(v_{k})$ , $v_{k}\stackrel{D^{\#}}{\rightharpoonup}0$
in $Y^{*}$, as a sequence in $X^{*}$, and let $v_{k}^{*}\in X$
be a dual conjugate of $v_{k}$. Consider a $\Delta$-profile decomposition
for $v_{k}^{*}$ in $X$. Then
\[
\|v_{k}\|_{X^{*}}=\sum_{n}\langle v_{k},g_{k}^{(n)}w^{(n)}\rangle_{X}+\langle v_{k},r_{k}\rangle_{X}\le\sum_{n}\langle g_{k}^{(n)*}v_{k},w^{(n)}\rangle_{X}+\|v_{k}\|_{Y^{*}}\|r_{k}\|_{Y}.
\]
It remains to observe that the sum in the right hand side is uniformly
convergent relative to $k$, and each term vanishes by the assumption
on $v_{k}$. The last term in the right hand side vanishes, since
$v_{k}$ is bounded in $Y^{*}$ and the remainder of profile decomposition
vanishes in $Y$.
\end{proof}
We can now prove Theorem \ref{thm:dual coco}.
\begin{proof}
Note first that condition \ref{eq:*-1} holds for $D_{0}^{\#}$ in
$Y^{*}$. Indeed, if $(g_{k}^{*})^{-1}\not\rightharpoonup0$, then
$\langle v,g_{k}^{-1}u\rangle\not\to0$ for some $u,v\in Y$, and
thus $g_{k}^{-1}\not\rightharpoonup0$ in $Y$, and, by density, $g_{k}^{-1}\not\rightharpoonup0$
in $X$. Then, by \ref{eq:*-1}, on a renamed subsequence, $g_{k}^{-1}\to g^{-1}$
in the strong operator sense in $X$ and, by imbedding, in $Y$. In
particular, $g^{-1}$ is an isometry and so also, by a simple duality
argument, is $(g^{*})^{-1}$. Then, for any $v\in Y^{*}$, ($g_{k}^{*})^{-1}v\rightharpoonup(g^{*})^{-1}v$,
and $\|(g_{k}^{*})^{-1}v\|_{Y^{*}}=\|(g^{*})^{-1}v\|_{Y^{*}}=\|v\|_{Y^{*}}$.
Since by assumption $Y^{*}$ is uniformly convex, we have ($g_{k}^{*})^{-1}\to(g^{*})^{-1}$
in the strong sense. 

It remains now to combine Proposition \ref{prop:dual coco} with Theorem
\ref{thm:main}, taking into account that weak and $\Delta$-convergence
of bounded sequences coincide by the Opial's condition. 
\end{proof}

\section{Profile decompositions: convergence of remainder}

We start this section with the proof of Theorem \ref{thm:main-1}.
\begin{proof}
By the Opial's condition $\Delta$-convergence in the uniformly convex
and uniformly smooth space $X$ is equivalent to the weak convergence
in $X$. Consequently, $D\in\mathcal{I}_{X}$. Moreover, $D$-$\Delta$-convergence
for bounded sequences coincides with $D$-weak convergence. Consequently,
since every bounded sequence in $X$ has a $\Delta$-profile decomposition
by Theorem \ref{thm:main}, it has a profile decomposition in the
sense of Definition \ref{def:pd}. 
\end{proof}
The rest of this section deals with general terms for interpretation
of $D,X$-weak of $D$-weak convergence as convergence in some norm.
In most cases this cannot be the norm of $X$, and verifying convergence
in a suitable weaker norm typically involves some hard analytic proof.
We give one example below where the group $D$ is sufficiently robust
to achieve convergence of $D$-weakly convergent sequences in the
norm of $X$. Our main concern, however, is cocompactness of imbeddings
of spaces of Sobolev type, which are discussed at the end of this
section. 
\begin{thm}
\label{thm:main-2}Let $X$ be a uniformly convex and uniformly smooth
Banach space and let $D$ be a nonempty subset of a dislocation group
$D_{0}$ on $X$. Then every bounded sequence $(u_{k})\subset X$
admits a $\Delta$-profile decomposition and (\ref{eq:Energy}) holds.
Furthermore, if $X$ is $X,D$ -cocompactly imbedded into a Banach
space $Y$, or if $X$ satisfies the Opial's condition and $X$ is
$D$-cocompactly imbedded into $Y$, then the remainder $r_{k}$ converges
to zero in the norm of $Y$.
\end{thm}
We also have a partial analog of this theorem that imposes some of
the assumptions on $Y$ instead of $X$.
\begin{thm}
\label{thm:main-2-1}Let $X$ be a Banach space densely imbedded into
a uniformly convex and uniformly smooth Banach space $Y$, and let
$D$ be a nonempty subset of a dislocation group $D_{0}$ on $Y$
such that with $D_{0}|_{X}$ is a dislocation group on $X$. $ $Then
every bounded sequence $(u_{k})\subset X$ admits a $\Delta$-profile
decomposition in $Y$ and (\ref{eq:Energy}) holds (in $Y$). Furthermore,
if there are only finitely many profiles $w^{(n)}\neq0$, $Y$ satisfies
the Opial's condition, and $X$ is $D$-cocompactly imbedded into
$Y$, then the remainder $r_{k}$ converges to zero in the norm of
$Y$.\end{thm}
\begin{proof}
Apply Theorem \ref{thm:main}in $Y$. Since $Y$ satisfies the Opial's
condition, all profiles are defined as weak limits in $Y$, and, since
$(x_{k})$ is bounded in $X$, they are elements of $X$. Since there
are only finitely many profiles, the remainder $r_{k}$ is bounded
in $X$. Since $r_{k}$ is bounded in $X$ and $r_{k}\stackrel{D}{\rightharpoonup}0$
in $Y$, we have also $r_{k}\stackrel{D}{\rightharpoonup}0$ in $X$,
and therefore, by $D$-cocompactness of the imbedding, $r_{k}\to0$
in $Y$.
\end{proof}
We now consider Besov and Triebel-Lizorkin spaces equipped with the
group of rescalings $D_{r}$, $r\in\mathbb{R},$ defined as the product
group of Euclidean shifts and, for some and dyadic dilations $g_{r}^{j}u(x)\mapsto2^{rj}u(2^{j}x)$,
$j\in\mathbb{Z}$. We refer to the definition in the book of Triebel
\cite{Triebel} (Definition 2, p.~238, see also a similar exposition
in Adams \& Fournier \cite{Adams}), based on the Littlewood-Paley
decomposition, of equivalent norm for Besov spaces $\dot{B}^{s,p,q}(\R^{N})$
and Triebel-Lizorkin spaces $\dot{F}^{s,p,q}(\R^{N})$. It is shown
by Cwikel, \cite{Cwikel-PD}, that for all $s\in\R$, and $p,q\in10,\infty)$
(i.e. when the corresponding spaces are uniformly convex and uniformly
smooth), the equivalent norm, which remains scale-invariant, satisfies
the Opial's condition. The latter work also gives direct proofs of
cocompactness of some of the imbeddings of Besov and Triebel-Lizorkin
spaces, which were implicitly proved, via wavelet argument, in \cite{BCK}.
We refer the reader to the survey \cite{survey} for explanations
why Assumption 1, verified in \cite{BCK} for Besov and Triebel-Lizorkin
spaces, implies cocompactness. We summarize the imbeddings whose cocompactness
is proved in \cite{BCK} (another proof, based on Littlewood-Paley
decomposition rather than on wavelet decomposition will be given in
a forthcoming paper \cite{Cwikel-PD}).
\begin{thm}
\label{thm:BCK}The following imbeddings are cocompact relative to
rescalings group $D_{N/p-s}$:\end{thm}
\begin{lyxlist}{00.00.0000}
\item [{(i)}] \textbf{$\dot{B}{}^{s,p,q}\hookrightarrow\dot{F}^{t,q,b}$},\textbf{
$\frac{1}{p}-\frac{1}{q}=\frac{s-t}{N}>0$}. 
\item [{(ii)}] \textbf{$\dot{B}{}^{s,p,a}\hookrightarrow\dot{B}^{t,q,b}$},\textbf{
$\frac{1}{p}-\frac{1}{q}=\frac{s-t}{N}\ge0$}, $a<b$. 
\item [{(iii)}] $\dot{B}^{s,p,p}\hookrightarrow\mathrm{BMO}$, $s=\frac{N}{p}>0$. 
\item [{(iv)}] \textbf{$\dot{B}{}^{s,p,a}\hookrightarrow L^{q,b}$},\textbf{
$\frac{1}{p}-\frac{1}{q}=\frac{s}{N}>0$}, $a<b$. 
\item [{(v)}] \textbf{$\dot{F}{}^{s,p,a}\hookrightarrow\dot{F}^{t,q,b}$},\textbf{
$\frac{1}{p}-\frac{1}{q}=\frac{s-t}{N}>0$}, $a,b>1$. 
\item [{(vi)}] \textbf{$\dot{F}{}^{s,p,a}\hookrightarrow\dot{B}^{t,q,p}$},\textbf{
$\frac{1}{p}-\frac{1}{q}=\frac{s-t}{N}>0.$} 
\end{lyxlist}

\section{Appendix A: Uniformly convex and uniformly smooth Banach spaces}
\begin{defn}
\label{def:muc}We recall that a normed vector space $X$ is called
uniformly convex if the following function on $[0,2]$, called the
\emph{modulus of convexity of $X$,} is strictly positive for all
$\epsilon>0$:
\[
\delta(\epsilon)=\inf_{x,y\in X,\,\|x\|=\|y\|=1,\,\|x-y\|=\epsilon}1-\left\Vert \frac{x+y}{2}\right\Vert .
\]

\end{defn}
As shown by Figiel (\cite{Figiel} Proposition 3, p.~122), the function
$\epsilon\mapsto\delta(\epsilon)/\epsilon$ is non-decreasing on $(0,2]$,
and thus $\epsilon\mapsto\delta(\epsilon)$ is strictly increasing
if $\delta(\epsilon)>0$ . Uniform convexity can be equivalently defined
by the property 

\begin{equation}
x,y\in X,\,\left\Vert x\right\Vert \le1,\,\left\Vert y\right\Vert \le1\,\Longrightarrow\,\left\Vert \frac{x+y}{2}\right\Vert \le1-\delta\left(\left\Vert x-y\right\Vert \right)\,,\label{eq:UC}
\end{equation}
(see \cite{Figiel} Lemma 4, p.~124.) 

It is an obvious consequence of (\ref{eq:UC}) that 
\begin{equation}
\left\Vert \frac{u+v}{2}\right\Vert \le\left\Vert v\right\Vert \left(1-\delta\left(\frac{\left\Vert u-v\right\Vert }{\left\Vert v\right\Vert }\right)\right)\label{eq:HyperUC}
\end{equation}
for any two elements $u,v\in X$ which satisfy $\left\Vert u\right\Vert \le\left\Vert v\right\Vert $
and $v\ne0$. This in turn implies that every two elements $u,v\in X$
which are not both zero satisfy
\begin{equation}
\left\Vert \frac{u+v}{2}\right\Vert \le C_{1}-C_{2}\delta\left(\frac{\left\Vert u-v\right\Vert }{C_{2}}\right)\mbox{ for all }C_{1}\mbox{ and }C_{2}\mbox{ in }[\max\left\{ \left\Vert u\right\Vert ,\left\Vert v\right\Vert \right\} ,\infty).\label{eq:UltraSuperUC}
\end{equation}
If $C_{1}=C_{2}=\max\left\{ \left\Vert u\right\Vert ,\left\Vert v\right\Vert \right\} $
then (\ref{eq:UltraSuperUC}) is exactly (\ref{eq:HyperUC}), possibly
with $u$ and $v$ interchanged. To extend this to larger values of
$C_{1}$ and $C_{2}$ we simply use the fact that $t\mapsto t\delta\left(\frac{\left\Vert u-v\right\Vert }{t}\right)$
is a non-increasing function.

A Banach space $X$ is called uniformly smooth if for every $\epsilon>0$
there exists $\delta>0$ such that if $x,y\in X$ with $\|x\|=1$
and $\|y\|\leq\delta$ then $\|x+y\|+\|x-y\|\le2+\epsilon\|y\|$.
It is known that $X^{*}$ is uniformly convex if and only if $X$
is uniformly smooth (\cite{LinTz}, Proposition 1.e.2) and that if
$X$ is uniformly convex, then the norm of $X$, as a function $\phi(x)=\|x\|$,
considered on the unit sphere $S_{1}=\{x\in X,\|x\|=1\}$, is uniformly
Gateau differentiable, which immediately implies that $\phi'$ is
a uniformly continuous function $S_{1}\to S_{1}^{*}$ (\cite{LinTz},
p. 61). Considering $\phi$ as a function on the whole $X$, one has
by homogeneity $\phi'(x)=\phi'(x/\|x\|)\in S_{1}^{*}$ for all $x\neq0$,
and an elementary argument shows that $\phi'(x)$ coincides with the
uniquely defined $x^{*}$. We summarize this characterization of the
duality conjugate as the following statement. 
\begin{lem}
\label{lem:dualconv}Let $X$ be a uniformly convex and uniformly
smooth Banach space. Then the map $x\mapsto x^{*}$ is a continuous
map $X\setminus\{0\}\to X^{*}$ with respect to the norm topologies
on $X$ and $X^{*}$ and is in fact uniformly continuous on all closed
subsets of $X\setminus\{0\}$. 
\end{lem}

\section{Appendix B: asymptotic centers and $\Delta$-convergence}

We follow the presentation of the Chebyshev and asymptotic centers
from Edelstein \cite{Edelstein}, in restriction to a particular case:
the objects in \cite{Edelstein} are defined there relative to a subset
$C$ of a Banach space $X$, and here we consider only the case $C=X$.
We follow the presentation of $\Delta$-convergence from Lim \cite{Lim}.

A bounded set $A$ in a Banach space $X$ can be assigned a positive
number

\[
R_{A}=\inf_{y\in X}\sup_{x\in A}\|x-y\|,
\]
called the Chebyshev radius of $A$. The Chebyshev radius is attained
(and is therefore a minimum) by weak lower semicontinuity of the norm
and the corresponding minimizer is called {\large{}the} Chebyshev
center of $A$. When $X$ is uniformly convex, the value $R_{A}$
cannot be attained at two different points $y'\neq y''$, since from
uniform convexity one immediately has $\sup_{x\in A}\|x-\frac{y'+y''}{2}\|<R_{A}$.
Consequently, the Chebyshev center of any set in a uniformly convex
space is unique. Theorem 1 in \cite{Edelstein} gives the following.
\begin{thm}
\label{thm:Edelstein}Let $X$ be a uniformly convex Banach space
and let $(x_{n})$ be a bounded sequence in $X$. Then the sequence
of Chebyshev centers $(y_{N})$ of the sets $A_{N}=(x){}_{k\ge N}$,
converges in norm. 
\end{thm}
By definition of the Chebyshev center of $A_{N}$, $N\in\N$, we have
$\sup_{k\ge N}\|x_{k}-y_{N}\|\le\sup_{k\ge N}\|x_{k}-y\|$ for all
$y\in X$, and the asymptotic center $y_{0}$ of the sequence $(x_{n})$
satisfies 
\begin{equation}
\limsup\|x_{n}-y_{0}\|\le\limsup\|x_{n}-y\|.\label{eq:ap}
\end{equation}
 From uniform convexity it follows immediately that
\begin{equation}
\limsup\|x_{n}-y_{0}\|<\limsup\|x_{n}-y\|,\; y\neq y_{0},\label{eq:assc}
\end{equation}
so the asymptotic center in a uniformly convex space is unique. An
equivalent definition of $\Delta$-limit in \cite[(2)]{Lim} says
that $y_{0}$ is a $\Delta$-limit of $(x_{n})$ if relation (\ref{eq:ap})
holds for every subsequence of $(x_{n})$. In particular, if a sequence
is $\Delta$-convergent, its $\Delta$-limit is also its asymptotic
center. On the other hand, an asymptotic center is not necessarily
the $\Delta$-limit. If, for example, $(x_{n})$ is an alternating
sequence of two points $a$ and $b$, its asymptotic center is $\frac{a+b}{2}$,
which is not a $\Delta$-limit limit of the sequence. 

The property of a space that every bounded sequence has an asymptotic
center is called in \cite{Lim} \emph{$\Delta$-completeness, }so
by \cite{Edelstein} , uniformly convex spaces are $\Delta$-complete.
Sequential $\Delta$-compactness follows from existence of a \emph{regular}
subsequence, i.e. a subsequence whose any further subsequence has
the same asymptotic radius. This is the content of \cite[ Lemma 15.2]{GK},
whose proof we reproduce below. Note that, unlike the proof of $\Delta$-compactness
in \cite{Lim}, no use is made of the Axiom of Choice.
\begin{proof}
Let $(x_{k})_{k\in\mathbb{N}}\subset X$ be a bounded sequence. We
use the notation $(v_{n})\prec(u_{n})$ to indicate that $(v_{n})$
is a subsequence of $(u_{n})$ and denote asymptotic radius of a sequence
$(v_{n})$ by $\rad n(v_{n})$. Set
\[
r_{0}=\inf\{\rad n(v_{n}):(v_{n})\prec(x_{n})\}.
\]
Select $(v_{n}^{1})\prec(x_{n})$ such that 
\[
\rad n(v_{n}^{1})<r_{0}+1
\]
 and let 
\[
r_{1}=\inf\{\rad n(v_{n}):(v_{n})\prec(v_{n}^{1})\}.
\]
Continuing by induction, and having defined $(v_{n}^{i})\prec(v_{n}^{i-1})$
set
\[
r_{i}=\inf\{\rad n(v_{n}):(v_{n})\prec(v_{n}^{i})\}
\]
and select $(v_{n}^{i+1})\prec(v_{n}^{i})$ so that 
\begin{equation}
\rad n(v_{n}^{i+1})<r_{i}+1/2^{i+1}.\label{eq:i+1}
\end{equation}
Note that $r_{0}\le r_{1}\le\dots$ so that $\lim_{i\to\infty}\rad n(v_{n}^{i})=r:=\lim r_{i}$.

Consider a diagonal sequence $(v_{k}^{k})$. Since $(v_{k}^{k})\prec(v_{k}^{i+1})$,
we have $\rad k(v_{k}^{k})\ge r$, while from (\ref{eq:i+1}) it follows
that $\rad k(v_{k}^{k})\le r$. Then $\rad k(v_{k}^{k})=r$, and since
the same argument applies to every subsequence of $(v_{k}^{k})$,
the sequence $(v_{k}^{k})$ is regular. 
\end{proof}

\section{Appendix C: an equivalent condition to \textup{(\ref{eq:*-1})}}

Condition (\ref{eq:*-1}), while it is verified in a great number
of applications, has a quite technical appearance. While we cannot
remedy this, we would like in this appendix to give it an equivalent
formulation. We will use the notation $\stackrel{s}{\to}$ for the
strong operator convergence.
\begin{prop}
Let $X$ be a uniformly convex separable Banach space and let $D_{0}$
be a group of isometries on $X$. \textup{Then condition (\ref{eq:*-1})
is equivalent to the following condition:}
\end{prop}
\begin{equation}
(g_{k})\subset D_{0},\; g_{k}\not\rightharpoonup0,\; u_{k}\rightharpoonup0\Rightarrow g_{k}u_{k}\rightharpoonup0\text{ on a subsequence}.\label{newii}
\end{equation}

\begin{lem}
\label{prop:XDisl}If $(g_{k})\subset D_{0}$ , $g_{k}\rightharpoonup g\neq0$,
is such that
\[
u_{k}\rightharpoonup0\Rightarrow g_{k}u_{k}\rightharpoonup0
\]
 then $g_{k}^{*}\stackrel{s}{\to}g^{*}$.\end{lem}
\begin{proof}
Let $v\in X^{*}$. We will verify that $g_{k}^{*}\stackrel{s}{\to}g^{*}$
once we show that for every bounded sequence $(u_{k})$, 
\[
\langle g_{k}^{*}v-g^{*}v,u_{k}\rangle\to0.
\]
Without loss of generality assume that $u_{k}\rightharpoonup u$.
Then

\[
\langle g_{k}^{*}v-g^{*}v,u_{k}\rangle=\langle g_{k}^{*}v-g^{*}v,u\rangle+\langle v,g_{k}(u_{k}-u)\rangle+\langle g^{*}v,(u_{k}-u)\rangle\to0,
\]
with the middle term vanishing by assumption and the remaining two
vanishing by the weak convergence.\end{proof}
\begin{lem}
\label{prop:verif}If $(g_{k})\subset D_{0}$, $g_{k}\rightharpoonup g\neq0$,
is such that
\[
u_{k}\rightharpoonup0\text{ in }X\Rightarrow g_{k}u_{k}\rightharpoonup0\text{ on a subsequence,}
\]
then
\[
v_{k}\rightharpoonup0\text{ in }X^{*}\Rightarrow g_{k}^{*}v_{k}\rightharpoonup0\text{ on a subsequence.}
\]
\end{lem}
\begin{proof}
By Lemma \ref{prop:XDisl} we have $g_{k}^{*}\stackrel{s}{\rightharpoonup}g^{*}$
and $g^{*}$ is necessarily a bijective isometry. This implies $g_{k}^{*}\rightharpoonup g^{*}$
and therefore $g_{k}\rightharpoonup g$, since $\langle v,(g_{k}-g)u\rangle=\langle(g_{k}^{*}-g^{*})v,u\rangle$.
At the same time, $\|g_{k}u\|=\|u\|=\|gu\|$, and thus, due to the
uniform convexity, $g_{k}u\to gu$, i.e. $\mbox{\ensuremath{g_{k}\stackrel{s}{\to}}g}$.
\end{proof}
Combining Lemma \ref{prop:XDisl} and Lemma \ref{prop:verif}, we
have the following statement.
\begin{lem}
\label{prop:4.5}If $(g_{k})\subset D_{0}$, $g_{k}\rightharpoonup g\neq0$,
is such that
\[
u_{k}\rightharpoonup0\Rightarrow g_{k}u_{k}\rightharpoonup0\text{ on a subsequence,}
\]
 then $g_{k}\stackrel{s}{\to}g$.
\end{lem}
We can now prove the proposition.
\begin{proof}
\emph{Necessity}. Relation (\ref{newii}) follows from (\ref{eq:*-1})
immediately. 

\emph{Sufficiency.} Let $(e_{n})$ be a normalized unconditional basis
in $X$. Since $\|g_{k}e_{n}\|=1$ for every $k$ and $n$, $(g_{k}e_{n})_{k}$
has a weakly convergent subsequence. By diagonalization and we easily
conclude that $(g_{k})$ has a weakly convergent subsequence. By assumption
the weak limit is nonzero. The sufficiency in the proposition follows
now from Lemma \ref{prop:4.5}.
\end{proof}

\subsection*{Acknowledgment }

The authors thank Michael Cwikel for bringing their attention to the
connection between $\Delta$-convergence and the works of Van Dulst,
Edelstein and Opial, for discussions, careful reading of an advanced
draft of this manuscript, and helpful editorial remarks.

\end{document}